\documentclass[10pt]{article}
\usepackage{amssymb,amsmath,comment}
\usepackage{enumerate} 
\usepackage{enumitem} 

\usepackage{color}

\newtheorem{theorem}{Theorem}[section]
\newtheorem{lemma}[theorem]{Lemma}

\newenvironment{proof}[1][]%
{\noindent {\setcounter{equation}{0}\it Proof.
}{#1}{}}{\hfill$\Box$\vspace{2ex}}

\def\longbox#1{\parbox{0.85\textwidth}{#1}}

\begin{document}
\title{Coloring square-free Berge graphs}

\author{%
Maria Chudnovsky\thanks{Mathematics Department, Princeton University, 
  NJ, USA. Partially supported by NSF grants DMS-1265803 and DMS-1550991.
This material is based upon work supported in part by the U. S. Army  Research 
Laboratory and the U. S. Army Research Office under    grant number 
W911NF-16-1-0404. }
\and%
Irene Lo\thanks{Department of Industrial Engineering and Operations
Research, Columbia University, New York, NY, USA.}
\and%
Fr\'ed\'eric Maffray\thanks{CNRS, Laboratoire G-SCOP, University of
Grenoble-Alpes, France. Partially supported by ANR project
STINT under reference ANR-13-BS02-0007}
\and%
Nicolas Trotignon\thanks{CNRS, LIP, ENS, Lyon, France.  Partially
supported by ANR project STINT under reference ANR-13-BS02-0007, and
by the LABEX MILYON (ANR-10-LABX-0070) of Universit\'e de Lyon, within
the program ‘‘Investissements d'Avenir’’ (ANR-11-IDEX-0007) operated
by the French National Research Agency (ANR)}
\and%
Kristina Vu\v{s}kovi\'c\thanks{School of Computing, University of
Leeds, Leeds LS2 9JT, UK; and Faculty of Computer Science (RAF), Union
University, Knez Mihailova 6/VI, 11000 Belgrade, Serbia.  Partially
supported by EPSRC grant EP/K016423/1 and Serbian Ministry of
Education and Science projects 174033 and III44006.}}

\date{\today}

 \maketitle

\begin{abstract}
We consider the class of Berge graphs that do not contain an induced 
cycle of length four.  We present a purely graph-theoretical algorithm
that produces an optimal coloring in polynomial time for every graph
in that class.
\end{abstract}


\section{Introduction}

A graph $G$ is \emph{perfect} if every induced subgraph $H$ of $G$
satisfies $\chi(H)=\omega(H)$, where $\chi(H)$ is the chromatic number
of $H$ and $\omega(H)$ is the maximum clique size in $H$.  In a graph
$G$, a \emph{hole} is an induced  cycle with at least four vertices
and an \emph{antihole} is the complement of a hole.  We say that graph
$G$ \emph{contains} a graph $F$, if $F$ is isomorphic to an induced
subgraph of $G$.  A graph $G$ is {\em $F$-free} if it does not contain
$F$, and for a family of graphs ${\cal F}$, $G$ is {\em ${\cal
F}$-free} if $G$ is $F$-free for every $F\in {\cal F}$.  Berge
{\cite{ber60,ber61,ber85}} introduced perfect graphs and conjectured
that a graph is perfect if and only if it does not contain an odd hole
or an odd antihole.  A \emph{Berge graph} is any graph that contains
no odd hole and no odd antihole.  This famous question (the Strong
Perfect Graph Conjecture) was solved by Chudnovsky, Robertson, Seymour
and Thomas \cite{CRST}: \emph{Every Berge graph is perfect}.
Moreover, Chudnovsky, Cornu\'ejols, Liu, Seymour and Vu\v{s}kovi\'c
\cite{CCLSV} devised a polynomial-time algorithm that determines if a
graph is Berge.

It is known that one can obtain an optimal coloring of a perfect graph
in polynomial time due to the algorithm of Gr\"otschel, Lov\'asz and
Schrijver \cite{GLS}.  This algorithm however is not purely
combinatorial and is usually considered impractical.  No purely
combinatorial algorithm exists for coloring all Berge graphs optimally
and in polynomial time.

The \emph{length} of a path or cycle is the number of its
edges. In what follows we will use the term ``path'' to mean ``induced
(or chordless) path''.
For a path $P$ with ends $a,b$, the {\em interior}
of $P$ is the set $V(P) \setminus \{a,b\}$; the interior of $P$ is
denoted by $P^*$.  We let $C_k$ denote the hole of length $k$ ($k\ge
4$).  The graph $C_4$ is also referred to as a {\em square}.  A graph
is \emph{chordal} if it is hole-free.  It is well-known that chordal
graphs are perfect and that their chromatic number can be computed in
linear time (see \cite{Gol}).

Farber \cite{Far89}, and later Alekseev \cite{Ale91}, proved that the
number of maximal cliques in a square-free graph on $n$ vertices is
$O(n^2)$.  Moreover it is known that one can list all the maximal
cliques in a graph $G$ in time $O(n^3K)$, where $K$ is the number of
maximal cliques; see \cite{TIAS,MU04} among others.  It follows that
finding $\omega(G)$ (the size of a maximum clique) can be done in
polynomial time for any square-free graph, and in particular finding
$\chi(G)$ can be done in polynomial time for a square-free Berge
graph.  Moreover, Parfenoff, Roussel and Rusu \cite{PRR} proved that
every square-free Berge graph has a vertex whose neigbhorhood is
chordal, which yields another way to find all maximal cliques in
polynomial time.  However getting an exact coloring of a square-free
Berge graph is still hard, and this is what we do.  The main result of
this paper is a purely graph-theoretical algorithm that produces an
optimal coloring for every square-free Berge graph in polynomial time.

\begin{theorem}\label{thm:main}
There exists an algorithm which, given any square-free Berge graph $G$
on $n$ vertices, returns a coloring of $G$ with $\omega(G)$ colors in
time $O(n^9)$.
\end{theorem}
 
A \emph{prism} is a graph that consists of two vertex-disjoint
triangles (cliques of size $3$) with three vertex-disjoint paths $P_1,
P_2, P_3$ between them, and with no other edge than those in the two
triangles and in the three paths.  Note that if two of $P_1, P_2, P_3$
have lengths of different parities, then their union induces an odd
hole.  So in a Berge graph, the three paths of a prism have the same
parity.  A prism is \emph{even} (resp.~\emph{odd}) if these three
paths all have even length (resp.~all have odd length).

Let $\cal A$ be the class of graphs that contain no odd hole, no
antihole of length at least $6$, and no prism.  This class is studied
in \cite{MT}, where purely graph-theoretical algorithms are devised
for coloring and recognizing graphs in that class.  In particular:
\begin{theorem}[\cite{MT}]\label{thm:MT}
There exists an algorithm which, given any graph $G$ in class ${\cal
A}$ on $n$ vertices, returns a coloring of $G$ with $\omega(G)$ colors
and a clique of size $\omega(G)$, in time $O(n^6)$.
\end{theorem}

Note that every antihole of length at least $6$ contains a
square; so a square-free graph contains no such antihole.  

Since Theorem~\ref{thm:MT} settles the case of graphs that have no
prism, we may assume for our proof of Theorem~\ref{thm:main} that we
are dealing with a graph that contains a prism.  The next sections
focus on the study of such graphs.  We will prove that whenever a
square-free Berge graph $G$ contains a prism, it contains a cutset of
a special type, and, consequently, that $G$ can be decomposed into two
induced subgraphs $G_1$ and $G_2$ such that an optimal coloring of $G$
can be obtained from optimal colorings of $G_1$ and $G_2$.

Note that results from~\cite{MTa} show that finding an induced prism
in a Berge graph can be done in polynomial time but that finding an
induced prism in general is NP-complete.

In \cite{M}, it was proved that when a square-free Berge graph
contains no odd prism, then either it is a clique or it has an ``even
pair'', as suggested by a conjecture of Everett and Reed (see
\cite{eps}).  However, this property does not carry over to all
square-free Berge graphs; indeed it follows from \cite{Houg} that the
line-graph of any $3$-connected square-free bipartite graph (for
example the ``Heawood graph'') is a square-free Berge graph with no
even pair.

\medskip

We finish this section with some notation and terminology.  In a graph
$G$, given a set $T\subset V(G)$, a vertex of $V(G) \setminus T$ is
\emph{complete to $T$} if it is adjacent to all vertices of $T$.  A
vertex of $V(G) \setminus T$ is \emph{anticomplete to $T$} if it is
non-adjacent to every vertex of $T$.  Given two disjoint sets
$S,T\subset V(G)$, $S$ is \emph{complete} to $T$ if every vertex of
$S$ is complete to $T$, and $S$ is \emph{anticomplete} to $T$ if every
vertex of $S$ is anticomplete to $T$.  Given a cycle, any
edge between two vertices that are not consecutive along it
is a \emph{chord}.  A cycle that has no chord is
\emph{chordless}.

The \emph{line-graph} of a graph $H$ is the graph $L(H)$ with
vertex-set $E(H)$ where $e,f\in E(H)$ are adjacent in $L(H)$ if they
share an end in $H$.

In a graph $J$, \emph{subdividing} an edge $uv\in E(J)$ means removing
the edge $uv$ and adding a new vertex $w$ and two new edges $uw,vw$.
Starting with a graph $J$, the effect of repeatedly subdividing edges
produces a graph $H$ called a \emph{subdivision} of $J$.  Note that
$V(J)\subseteq V(H)$.  A \emph{bipartite subdivision} of a graph $J$
is any subdivision of $J$ that is bipartite.

\begin{lemma}\label{C4cliques}
Let $G$ be square-free.  Let $K$ be a clique in $G$, possibly empty.
Let $X_1, X_2, \ldots, X_k$ be pairwise disjoint subsets of $V(G)$,
also disjoint from $K$, such that $X_i$ is complete to $X_j$ for all
$i\neq j$, and let $X = \bigcup_i X_i$.  Suppose that for every $v$ in
$K$, there is an integer $i$ so that $v$ is complete to $X\setminus
X_i$.  Then there is an integer $i$ such that $(K\cup X)\setminus X_i$
is a clique in $G$.
\end{lemma}
\begin{proof}
First observe that there exists an integer $j$ such that $X\setminus
X_j$ is a clique, for otherwise two of $X_1, \ldots, X_k$ are not
cliques and their union contains a square.  Hence if $K$ is empty, the
lemma holds.  

Now we claim that $K$ is complete to at least $k-1$ of the $X_i$'s.
For suppose on the contrary that $K$ is not complete to any of $X_1$
and $X_2$.  Then there are vertices $v_1, v_2 \in K$, $x_1 \in X_1$,
$x_2 \in X_2$ such that for $i \in \{1,2\}$ and $j \in
\{1,2\}\setminus \{i\}$, $v_i$ adjacent to $x_i$ and non-adjacent to
$x_j$.  By the assumption, $v_1\neq v_2$.  Then $\{v_1,x_1,x_2,v_2\}$
induces a square, contradiction.  Hence there exists an index $h$ such
that $K$ is complete to $X\setminus X_h$.
	
Suppose that the lemma does not hold.  Then $j \neq h$ and there are
vertices $x, x' \in X_j$, $v \in K$, $w \in X_h$ such that $x$ and
$x'$ are non-adjacent and $v$ and $w$ are non-adjacent.  Then
$\{x,v,x',w\}$ induces a square, contradiction.  This proves the lemma.
\end{proof}

In a graph $G$, we say (as in \cite{CRST}) that a vertex $v$ can be
\emph{linked} to a triangle $\{a_1,a_2,a_3\}$ (via paths
$P_1,P_2,P_3$) when: the three paths $P_1,P_2,P_3$ are mutually
vertex-disjoint; for each $i\in\{1,2,3\}$, $a_i$ is an end of $P_i$;
for all $i,j\in\{1,2,3\}$ with $i\neq j$, $a_ia_j$ is the only edge
between $P_i$ and $P_j$; and $v$ has a neighbor in each of
$P_1,P_2,P_3$.
\begin{lemma}[(2.4) in \cite{CRST}]\label{lem:link}
In a Berge graph, if a vertex $v$ can be linked to a triangle
$\{a_1,a_2,a_3\}$, then $v$ is adjacent to at least two of
$a_1,a_2,a_3$.
\end{lemma}

\section{Good partitions}

In a graph $G$, a \emph{triad} is a set of three pairwise 
non-adjacent vertices.

A \emph{good partition} of a graph $G$ is a partition
$(K_1,K_2,K_3,L,R)$ of $V(G)$ such that:
\begin{itemize}
\item[(i)] 
$L$ and $R$ are not empty, and $L$ is anticomplete to $R$;
\item[(ii)] 
$K_1\cup K_2$ and $K_2\cup K_3$ are cliques;
\item[(iii)] 
  If $P=p_1-\dots-p_k$ is a path with $p_1 \in K_1$, $p_k \in K_3$, $k \geq 3$
  and  $P^* \subseteq L$, then $p_2$ is complete to $K_1$. 
\item[(iv)] 
Either $K_1$ is anticomplete to $K_3$, or for every $v \in L$
the set $N(v) \cap K_1$ is complete to $N(v) \cap K_3$;
\item[(v)] 
For some $x\in L$ and $y\in R$, there is a triad of $G$ that contains
$\{x,y\}$.
\end{itemize}

\begin{theorem}
Let $G$ be a square-free Berge graph.  If $G$ contains a prism, then
$G$ has a good partition.
\end{theorem}
The proof of this theorem will be given in the following sections,
depending on the presence in $G$ of an even prism
(Theorem~\ref{thm:epr}), an odd prism (Theorem~\ref{thm:opr}), or the
line-graph of a bipartite subdivision of $K_4$
(Theorem~\ref{thm:lgb}).

\medskip

In the rest of this section we show how a good partition can be used
to find an optimal coloring of the graph.

\begin{lemma}\label{lem:cutset}
Let $G$ be a square-free Berge graph.  Suppose that $V(G)$ has a good
partition $(K_1, K_2, K_3, L, R)$.  Let $G_1=G\setminus R$ and
$G_2=G\setminus L$, and for $i=1,2$ let $c_i$ be an
$\omega(G_i)$-coloring of $G_i$.  Then an $\omega(G)$-coloring of $G$
can be obtained in polynomial time.
\end{lemma}
\begin{proof}
We may assume (by making $K_2$ maximal) that no vertex of $K_3$ is
complete to $K_1$.  Since $K_1\cup K_2$ is a clique, by permuting
colors we may assume that $c_1(x)=c_2(x)$ holds for every vertex $x\in
K_1\cup K_2$.

Say that a vertex $u$ in $K_3$ is \emph{bad} if $c_1(u)\neq c_2(u)$,
and let $B$ be the set of bad vertices.  If $B=\emptyset$, we can
merge $c_1$ and $c_2$ into a coloring of $G$ and the lemma holds.
Therefore let us assume that $B\neq\emptyset$.  We will show that we
can produce in polynomial time a pair $(c'_1, c'_2)$ of
$\omega(G)$-colorings of $G_1$ and $G_2$, respectively, that agree on
$K_1\cup K_2$ and have strictly fewer bad vertices than $(c_1, c_2)$.
Repeating this argument at most $|B|$ times will prove the lemma.

For each $h\in\{1,2\}$ and for any two distinct colors $i$ and $j$,
let $G^{i,j}_h$ be the bipartite subgraph of $G_h$ induced by $\{v\in
V(G_h) \mid c_h(v)\in\{i,j\}\}$; and for any vertex $u\in K_3$, let
$C^{i,j}_h(u)$ be the component of $G^{i,j}_h$ that contains $u$.

Let $u\in B$, with $i=c_1(u)$ and $j=c_2(u)$.  Then $C^{i,j}_h(u) \cap
K_2=\emptyset$ for each $h\in\{1,2\}$ because $u$ is complete to
$K_2$.  Say that $u$ is \emph{free} if $C^{i,j}_h(u)\cap
K_1=\emptyset$ holds for some $h\in\{1,2\}$.  In particular $u$ is
free whenever colors $i$ and $j$ do not appear in $K_1$.
\begin{equation}\label{eq:nofree}
\mbox{We may assume that there is no free vertex.}
\end{equation}
Suppose that $u$ is a free vertex, with $C^{i,j}_1(u)\cap K_1=
\emptyset$ say.  Then we swap colors $i$ and $j$ on $C^{i,j}_1(u)$.
We obtain a coloring $c'_1$ of $G_1$ where the color of every vertex
in $K_1\cup K_2$ is unchanged, by the definition of a free vertex; so
$c'_1$ and $c_2$ agree on $K_1\cup K_2$.  For all $v\in K_3\setminus
B$ we have $c_1(v)\neq i$, because $c_1(u)=i$, and $c_1(v)\neq j$,
because $c_1(v)=c_2(v)\neq c_2(u)=j$; so the color of $v$ is
unchanged.  Moreover we have $c'_1(u)=j=c_2(u)$, so $c'_1$ and $c_2$
agree on $u$.  Hence the pair $(c'_1,c_2)$ has strictly fewer bad
vertices than $(c_1,c_2)$.  Thus (\ref{eq:nofree}) holds.

\medskip

Choose $w$ in $B$ with the largest number of neighbors in $K_1$. Then:
\begin{equation}\label{cl1}
\mbox{Every vertex $u\in B$ satisfies $N(u)\cap K_1\subseteq N(w)\cap
K_1$.}
\end{equation}
For suppose that some vertex $x\in K_1$ is adjacent to $u$ and not to
$w$.  By the choice of $w$ there is a vertex $y\in K_1$ that is
adjacent to $w$ and not to $u$.  Then $\{x, y, u, w\}$ induces a
square, a contradiction.  Thus (\ref{cl1}) holds.

\medskip

Up to relabelling, let $c_1(w)=1$ and $c_2(w)=2$.  By
(\ref{eq:nofree}) $w$ is not a free vertex, so $C_1^{1,2}(w)\cap
K_1\neq\emptyset$ and $C_2^{1,2}(w)\cap K_1\neq\emptyset$.  Hence for
some $i\in\{1,2\}$ there is a  path
$P=w$-$p_1$-$\cdots$-$p_k$-$a$ in $C^{1,2}_1(w)$, with $k\ge 1$,
$p_1\in K_3\cup L$, $p_2, \ldots, p_k\in L$ and $a\in K_1$ with
$c_1(a)=i$; and for some $i'\in\{1,2\}$ there is a  path
$Q=w$-$q_1$-$\cdots$-$q_\ell$-$a'$ in $C^{1, 2}_2(w)$, with $\ell\ge
1$, $q_1\in K_3\cup R$, $q_2, \ldots, q_\ell\in R$ and $a'\in K_1$
with $c_2(a')={i'}$.  It follows that at least one of the colors $1$
and $2$ appears in $K_1$.  We claim that:
\begin{equation}\label{a1a2}
\mbox{Exactly one of the colors $1$ and $2$ appears in $K_1$.}
\end{equation}
For suppose that there are vertices $a_1,a_2\in K_1$ with $c_1(a_1)=1$
and $c_1(a_2)=2$.  We know that $w$ is anticomplete to $\{a_1,a_2\}$.
Since $P$ is bicolored by $c_1$, it cannot contain a vertex complete
to $\{a_1,a_2\}$; so, by assumption (iii), $P$ does not meet $L$.
This implies that $P=w$-$p_1$-$a_1$ and $p_1\in K_3$.  Then
$c_2(p_1)\neq 2$, because $c_2(w)=2$, and so $p_1\in B$; but then
(\ref{cl1}) is contradicted since $w$ is non-adjacent to $a_1$.  Thus
(\ref{a1a2}) holds.

\medskip

By (\ref{a1a2}) we have $i=i'$ and $a=a'$.  Let $j=3-i$.  Note that if
$i=1$ then $P$ has even length and $Q$ has odd length, and if $i=2$
then $P$ has odd length and $Q$ has even length.  So $P$ and $Q$ have
different parities.  If $p_1\in L$ and $q_1\in R$, then $V(P)\cup
V(Q)$ induces an odd hole, a contradiction.  Hence,
\begin{equation}\label{p1orq1inK3}
\mbox{At least one of $p_1$ and $q_1$ is in $K_3$.}
\end{equation}  
We claim that:
\begin{equation}\label{eq:nob21}
\mbox{There is no vertex $y$ in $K_3$ such that $c_1(y)=2$ and
$c_2(y)=1$.}
\end{equation}
Suppose that there is such a vertex $y$.  If $p_1\in K_3$ and $q_1\in
K_3$, then $p_1=y=q_1$ and $(V(P)\cup V(Q))\setminus\{w\}$ induces an
odd hole, a contradiction.  So, by (\ref{p1orq1inK3}), exactly one of
$p_1$ and $q_1$ is in $K_3$.  Suppose that $p_1\in K_3$ and $q_1\in
R$.  So $p_1=y$, and in particular $p_1a$ is not an edge.  If $p_1$
has no neighbor on $Q\setminus w$, then $V(P)\cup V(Q)$ induces an odd
hole.  So suppose that $p_1$ has a neighbor on $Q\setminus w$.  Then
there is a path $Q'$ from $p_1$ to $a'$ with interior in $Q\setminus
w$, and since it is bicolored by $c_2$ the parity of $Q'$ is different
from the parity of $Q$.  Then $(V(P)\setminus\{w\})\cup V(Q')$ induces
an odd hole, a contradiction.  When $p_1\in L$ and $q_1\in K_3$ the
proof is similar.  Thus (\ref{eq:nob21}) holds.

\begin{equation}\label{eq:nob2} 
p_1\notin K_3.
\end{equation}
For suppose that $p_1\in K_3$.  We have $c_2(p_1)\neq 1$ by
(\ref{eq:nob21}) and $c_2(p_1)\neq 2$ because $c_2(w)=2$.  Hence let
$c_2(p_1)=3$.  So color~$3$ does not appear in $K_2$.  Note that every
vertex in $K_3\setminus B$ has its color different from $1,2,3$
because of $w$ and $p_1$.

Suppose that color~$3$ does not appear in $K_1$.  Then, by
(\ref{a1a2}), $C^{j,3}_2(p_1)\cap K_1=\emptyset$.  We swap colors $j$
and $3$ on $C^{j,3}_2(p_1)$.  We obtain a coloring $c_2'$ of $G_2$
such that the color of all vertices in $K_1\cup K_2$ is unchanged, so
$c'_2$ agrees with $c_1$ on $K_1\cup K_2$.  For every vertex $v$ in
$K_3\setminus B$ we have $c_2(v)\neq 3$, because $c_2(p_1)=3$, and
$c_2(v)\notin \{1,2\}$, because $c_2(v)=c_1(v)$ and $\{1,2\}=
\{c_1(w),c_1(p_1)\}$; so $c'_2(v)=c_2(v)$.  Moreover, $c_2'(p_1)=j$.
If $j=1$, then $c'_2(w)=c_2(w)$ and the pair $(c_1, c_2')$,
contradicts (\ref{eq:nob21}) (with $y=p_1$).  If $j=2$, then
$c'_2(p_1)= c_1(p_1)$, so the pair $(c_1,c_2')$ has strictly fewer bad
vertices than $(c_1,c_2)$.  Therefore we may assume that there is a
vertex $a_3$ in $K_1$ with $c_1(a_3)=3$.

Vertex $p_1$ is not adjacent to $a_3$ because $c_2(p_1)=c_2(a_3)$, and
$p_1$ is not adjacent to $a$ by (\ref{cl1}) and because $w$ is not
adjacent to $a$.  This implies $k\ge 2$, so the path $P\setminus w$
meets $L$.  Assumption (iii) implies that $P\setminus w$ contains a
vertex that is complete to $\{a,a_3\}$, and since $P$ is a path,
that vertex is~$p_k$.

Suppose that $a_3$ has a neighbor $p_g$ on $P\setminus\{w,p_k\}$, and
choose the smallest such integer $g$.  We know that $g\ge 2$.  The
path $p_1$-$\cdots$-$p_g$-$a_3$ meets $L$, but it contains
no vertex that is complete to $\{a,a_3\}$ because $a$ has no neighbor
on $P\setminus p_k$, so assumption (iii) is contradicted.  Therefore
$a_3$ has no neighbor on $P\setminus \{w,p_k\}$.

Suppose that $i=1$.  Then $P$ has even length, and by (\ref{a1a2})
color $2$ does not appear in $K_1$.  If $w$ is adjacent to $a_3$, then
since $k\geq 2$, we see that $(V(P)\setminus \{a\})\cup \{a_3\}$
induces an odd hole.  So $w$ is non-adjacent to $a_3$.  Hence
$\{a,a_3\}$ is anticomplete to $\{w,p_1\}$.  Since, by
(\ref{eq:nofree}), $p_1$ is not a free vertex, and color~$2$ does not
appear in $K_1$, there is a path $S$ between $p_1$ and $a_3$
in $C^{2,3}_2(p_1)$, and $S$ has even length because $c_2(p_1)=
c_2(a_3)$.  If $w\in V(S)$, then $V(S)\cup\{p_1,\ldots,p_k\}$ induces
an odd hole.  If $w\notin V(S)$, then the interior of $S$ is in $R$,
and $V(S)\cup\{p_2,\ldots,p_k\}$ induces an odd hole.

Now suppose that $i=2$.  By (\ref{eq:nofree}), $p_1$ is not a free
vertex, so there is a path $T$ from $p_1$ to $\{a, a_3\}$ in
$C^{2,3}_1(p_1)$.  Since $T$ is bicolored by $c_1$ it cannot contain a
vertex that is complete to $\{a,a_3\}$, so assumption (iii) implies
that $T$ does not meet $L$.  So we have $T=p_1$-$x$-$a$ for some
vertex $x$ in $K_3$ with $c_1(x)=3$.  We have $c_2(x)\neq 3$ because
$c_2(p_1)=3$; so $x\in B$.  But the fact that $a$ is adjacent to $x$
and not to $w$ contradicts (\ref{cl1}).  Thus (\ref{eq:nob2}) holds.

\medskip

By (\ref{p1orq1inK3}) and (\ref{eq:nob2}) we have $p_1\not\in K_3$ and
$q_1\in K_3$.  In particular, $P$ meets $L$.  We have $c_1(q_1)\neq 1$
because $c_1(w)=1$, and $c_1(q_1)\neq 2$ by (\ref{eq:nob21}).  Hence
let $c_1(q_1)=3$.  Note that every vertex in $K_3\setminus B$ has its
color different from $1,2,3$ because of $w$ and $q_1$.
\begin{equation}\label{eq:c3inK1}
\mbox{Color~$3$ appears in $K_1$.}
\end{equation}
Assume the contrary.  Then, by (\ref{a1a2}), $C^{j,3}_1(q_1)\cap
K_1=\emptyset$.  We swap colors $j$ and $3$ on $C^{j,3}_1(q_1)$.  We
obtain a coloring $c_1'$ of $G_1$ such that the color of every vertex
in $K_1\cup K_2$ is unchanged, so $c'_1$ agrees with $c_2$ on $K_1\cup
K_2$.  Also every vertex $v$ in $K_3\setminus B$ satisfies $c'_1(v)=
c_1(v)$.  Moreover, $c_1'(q_1)=j$.  If $j=1$, then $c_1'(q_1)=
c_2(q_1)$, so the pair $(c_1',c_2)$ has strictly fewer bad vertices
than $(c_1,c_2)$.  If $j=2$, then $c'_1(w)=c_1(w)=1$, so the pair
$(c_1',c_2)$ contradicts (\ref{eq:nob21}) (with $y=q_1$).  Thus we may
assume that (\ref{eq:c3inK1}) holds.

\medskip

By (\ref{eq:c3inK1}) there is a vertex $a_3$ in $K_1$ with $c_1(a_3)=
3$.  By (\ref{cl1}), $q_1$ is anticomplete to $\{a,a_3\}$.  Vertex
$q_1$ has a neighbor in $P\setminus w$, for otherwise $V(P)\cup V(Q)$
induces an odd hole.  So there is a path $P'$ from $q_1$ to
$a$ with interior in $P\setminus w$, and $P'$ meets $L$ because it
contains $p_k$.  By assumption (iii), $P'$ contains a vertex that is
complete to $\{a,a_3\}$, and since $P$ is a path  that vertex is
$p_k$.

Suppose that $C^{i,3}_1(q_1)\cap K_1=\emptyset$.  Then we swap colors
$i$ and $3$ on $C^{i,3}_1(q_1)$.  We obtain a coloring $c_1'$ of $G_1$
such that the color of every vertex in $K_1\cup K_2$ is unchanged, so
$c'_1$ agrees with $c_2$ on $K_1\cup K_2$.  Also every vertex $v$ in
$K_3\setminus B$ satisfies $c'_1(v)=c_1(v)$.  Moreover, $c_1'(q_1)=i$.
If $i=1$, the pair $(c_1',c_2)$ has strictly fewer bad vertices than
$(c_1, c_2)$.  If $i=2$, then $c'_1(w)=c_1(w)=1$, so $(c_1',c_2)$
contradicts (\ref{eq:nob21}) (with $y=q_1$).  Therefore we may assume
that $C^{i,3}_1(q_1)\cap K_1\neq \emptyset$.

Let $Z$ be a path from $q_1$ to $\{a,a_3\}$ in $C^{i,3}_1(q_1)$.
Since $Z$ is bicolored by $c_1$, no vertex of $Z$ can be complete to
$\{a,a_3\}$, and so assumption (iii) implies that $Z$ does not meet
$L$.  This means that either $i=1$ and $Z=q_1$-$w$-$a_3$, or $i=2$ and
$Z=q_1$-$z$-$a_3$ for some $z$ in $K_3$ with $c_1(z)=2$.  In either
case, $K_1$ is not anticomplete to $K_3$.  Since $p_k$ is complete to
$K_1$, and no vertex of $K_3$ is complete to $K_1$, assumption (iv)
implies that $p_k$ is anticomplete to $K_3$.  In particular, $p_k$ is
non-adjacent to~$q_1$ and $k\geq 2$.

If $a_3$ has a neighbor in $P\setminus\{w,p_k\}$, then (since $q_1$
also has a neighbor in $P\setminus\{w,p_k\}$) there is a 
path from $q_1$ to $a_3$ with interior in $V(P)\setminus\{w,p_k\}$,
so, by (iii), that path must contain a vertex that is complete to
$\{a,a_3\}$; but this is impossible because $a$ has no neighbor in
$P\setminus p_k$.  So $a_3$ has no neighbor in $P\setminus\{w,p_k\}$.

Now if $i=1$, then $w$ is adjacent to $a_3$, and $P$ has even length,
hence $(V(P)\setminus\{a\}) \cup \{a_3\}$ induces an odd hole.  So
$i=2$, and $Z=q_1$-$z$-$a_3$ with $z\in K_3$ and $c_1(z)=2$.  Recall
that $p_k$ is non-adjacent to $z$.  The path $z$-$w$-$P$-$p_k$ has odd
length, and it is bicolored by $c_1$, so it contains an odd
path $P''$ from $z$ to $p_k$.  But then $V(P'')\cup\{a_3\}$ induces an
odd hole.  This completes the proof.
\end{proof}

\section{Prisms and hyperprisms}

In a graph $G$ let $R_1,R_2,R_3$ be three  paths that form a
prism $K$ with triangles $\{a_1, a_2, a_3\}$ and $\{b_1, b_2, b_3\}$,
where each $R_i$ has ends $a_i$ and $b_i$.  A vertex of $V(G)\setminus
K$ is a \emph{major neighbor} of $K$ if it has at least two neighbors
in $\{a_1,a_2,a_3\}$ and at least two neighbors in $\{b_1,b_2,b_3\}$.
A subset $X$ of $V(K)$ is \emph{local} if either
$X\subseteq\{a_1,a_2,a_3\}$ or $X\subseteq\{b_1,b_2,b_3\}$ or
$X\subseteq V(R_i)$ for some $i\in\{1,2,3\}$.

If $F,K$ are induced subgraphs of $G$ with $V(F)\cap V(K)=\emptyset$,
any vertex in $K$ that has a neighbor in $F$ is called an
\emph{attachment} of $F$ in $K$, and whenever any such vertex exists
we say that $F$ \emph{attaches} to $K$.

Here are several theorems from the Strong Perfect Graph
Theorem~\cite{CRST} that we will use.

\begin{theorem}[(7.4) in \cite{CRST}]\label{spgt74}
In a Berge graph $G$, let $R_1,R_2,R_3$ be three  paths, of
even lengths, that form a prism $K$ with triangles $\{a_1, a_2, a_3\}$
and $\{b_1, b_2, b_3\}$, where each $R_i$ has ends $a_i$ and $b_i$.
Assume that $R_1, R_2, R_3$ all have length at least $2$.  Let $R'_1$
be a  path from $a'_1$ to $b_1$, such that $R'_1, R_2, R_3$
also form a prism.  Let $y$ be a major neighbor of $K$.  Then $y$ has
at least two neighbors in $\{a'_1, a_2, a_3\}$.
\end{theorem}

\begin{theorem}[(10.1) in \cite{CRST}]\label{spgt101}
In a Berge graph $G$, let $R_1,R_2,R_3$ be three  paths that
form a prism $K$ with triangles $\{a_1, a_2, a_3\}$ and $\{b_1, b_2,
b_3\}$, where each $R_i$ has ends $a_i$ and $b_i$.  Let $F \subseteq
V(G) \setminus V(K)$ be connected, such that its set of attachments in
$K$ is not local.  Assume no vertex in $F$ is major with respect to
$K$.  Then there is a path $f_1$-$\ldots$-$f_n$ in $F$ with $n \ge 1$,
such that (up to symmetry) either:
\begin{itemize}
\item[1.]  
$f_1$ has two adjacent neighbors in $R_1$, and $f_n$ has two adjacent
neighbors in $R_2$, and there are no other edges between $\{f_1,
\ldots , f_n\}$ and $V(K)$, and (therefore) $G$ has an induced
subgraph which is the line graph of a bipartite subdivision of $K_4$,
or 
\item[2.]  
$n \ge 2$, $f_1$ is adjacent to $a_1, a_2, a_3$, and $f_n$ is adjacent
to $b_1, b_2, b_3$, and there are no other edges between $\{f_1,
\ldots, f_n\}$ and $V(K)$, or  
\item[3.]  
$n \ge 2$, $f_1$ is adjacent to $a_1, a_2$, and $f_n$ is adjacent to
$b_1, b_2$, and there are no other edges between $\{f_1, \ldots,
f_n\}$ and $V(K)$, or 
\item[4.] 
$f_1$ is adjacent to $a_1, a_2$, and there is at least one edge
between $f_n$ and $V(R_3) \setminus \{a_3\}$, and there are no other
edges between $\{f_1, \ldots, f_n\}$ and $V(K) \setminus \{a_3\}$.
\end{itemize}
\end{theorem}


\subsubsection*{Hyperprisms}
A \emph{hyperprism} is a graph $H$ whose vertex-set can be partitioned
into nine~sets:
\begin{center}
\begin{tabular}{ccc}
$A_1$ & $C_1$ & $B_1$ \\ 
$A_2$ & $C_2$ & $B_2$ \\ 
$A_3$ & $C_3$ & $B_3$
\end{tabular}
\end{center}
with the following properties:
\begin{itemize}
\item
Each of $A_1,A_2,A_3,B_1,B_2,B_3$ is non-empty.
\item
For distinct $i,j\in\{1,2,3\}$, $A_i$ is complete to $A_j$, and $B_i$
is complete to $B_j$, and there are no other edges between $A_i \cup
B_i \cup C_i$ and $A_j \cup B_j \cup C_j$.
\item
For each $i\in\{1,2,3\}$, every vertex of $A_i \cup B_i \cup C_i$
belongs to a  path between $A_i$ and $B_i$ with interior in
$C_i$.
\end{itemize}
For each $i\in\{1,2,3\}$, any  path from $A_i$ to $B_i$ with
interior in $C_i$ is called an \emph{$i$-rung}.  The triple $(A_i,
C_i, B_i)$ is called a \emph{strip} of the hyperprism.  If we pick any
$i$-rung $R_i$ for each $i\in\{1,2,3\}$, we see that $R_1, R_2, R_3$
form a prism; any such prism is called an \emph{instance} of the
hyperprism.  If $H$ contains no odd hole, it is easy to see that all
rungs have the same parity; then the hyperprism is called even or odd
accordingly.  Note that if $H$ is an even hyperprism, then $A_i$ is
anticomplete to $B_i$ for each $i$.  On the other hand, if $H$ is an
odd hyperprism, there may be edges between $A_i$ and $B_i$ for any
$i$; however, if $H$ is square-free there is at most one integer $i$
such that there is an edge between $A_i$ and $B_i$.

Let $G$ be a graph that contains a prism.  Then $G$ contains a
hyperprism $H$.  Let $(A_1, \ldots, B_3)$ be a partition of $V(H)$ as
in the definition of a hyperprism above.  A subset $X \subseteq V(H)$
is \emph{local} if either $X\subseteq A_1\cup A_2\cup A_3$ or
$X\subseteq B_1\cup B_2\cup B_3$ or $X\subseteq A_i\cup B_i\cup C_i$
for some $i\in\{1,2,3\}$.  A vertex $x$ in $V(G)\setminus V(H)$ is a
\emph{major neighbor} of $H$ if $x$ is a major neighbor of some
instance of $H$.  The hyperprism $H$ is \emph{maximal} if there is no
hyperprism $H'$ such that $V(H)$ is strictly included in $V(H')$.

\begin{lemma}\label{lem:hyplocal}
Let $G$ be a Berge graph, let $H$ be a hyperprism in $G$, and let $M$
be the set of major neighbors of $H$ in $G$.  Let $F$ be a component
of $G\setminus (V(H)\cup M)$ such that the set of attachments of $F$
in $H$ is not local.  Then one can find in polynomial time one of the
following:
\begin{itemize}
\item 
A  path $P$, with $\emptyset\neq V(P)\subseteq V(F)$, such
that $V(H)\cup V(P)$ induces a hyperprism (of the same parity as $H$).
\item 
A  path $P$, with $\emptyset\neq V(P)\subseteq V(F)$, and for
each $i\in\{1,2,3\}$ an $i$-rung $R_i$ of $H$, such that $V(P)\cup
V(R_1)\cup V(R_2)\cup V(R_3)$ induces the line-graph of a bipartite
subdivision of $K_4$.
\end{itemize}
\end{lemma}
\begin{proof}
When $H$ is an even hyperprism, the proof of the lemma is identical to
the proof of Claim~(2) in the proof of Theorem~10.6 in \cite{CRST},
and we omit it.  When $H$ is an odd hyperprism, the proof of the lemma
is similar to the proof of Claim~(2), with the following adjustments:
the case when the integer $n$ in that proof is even and the case when
$n$ is odd are swapped, and the argument on page 126 of \cite{CRST},
lines 16--18, is replaced with the following argument:

\begin{quote}
Suppose that $f_n$ is not adjacent to $b_1$; so $f_1$ is adjacent to
$b_1$, $n\ge 2$, and $f_n$ is adjacent to $a_2$.  Let $R_3$ be any
$3$-rung, with ends $a_3\in A_3$ and $b_3\in B_3$.  Then $a_1b_1$ is
an edge, for otherwise $f_1$-$a_1$-$R_1$-$b_1$-$f_1$ is an odd hole;
and $f_1$ has no neighbor in $\{a_3,b_3\}$, for otherwise we would
have $n=1$.  Likewise, $a_2b_2$ is an edge, and $f_n$ has no neighbor
in $\{a_3,b_3\}$.  But then $V(R_1)\cup V(R_2)\cup V(R_3)\cup\{f_1,
\ldots, f_n\}$ induces the line-graph of a bipartite subdivision of
$K_4$, a contradiction.
\end{quote}

This completes the proof of the lemma.  
\end{proof}

\section{Even prisms}

We need to analyze the behavior of major neighbors of an even
hyperprism.  The reader may note that in the following theorem we are
not assuming that the graph is square-free.

\begin{theorem}\label{thm:ehyp}
Let $G$ be a Berge graph that contains an even prism and does not
contain the line-graph of a bipartite subdivision of $K_4$.  Let $H$
be an even hyperprism in $G$, with partition $(A_1,\ldots,B_3)$
as in the definition of a hyperprism, and let $x$ be a major neighbor
of $H$.  Then either: 
\begin{itemize}
\item $x$ is complete to at least two of $A_1, A_2, A_3$ and at least
two
of $B_1, B_2, B_3$, or 
\item $V(H)\cup\{x\}$ induces a hyperprism.
\end{itemize}
\end{theorem}
\begin{proof}
Since $x$ is a major neighbor of $H$, there exists for each $i\in\{1,
2, 3\}$ an $i$-rung $W_i$ of $H$ such that $x$ is a major neighbor of
the prism $K_W$ formed by $W_1, W_2, W_3$.  Suppose that the first
item does not hold; so, up to symmetry, $x$ has a non-neighbor $u_1\in
A_1$ and a non-neighbor $u_2\in A_2$.  For each $i\in\{1,2\}$ let
$U_i$ be an $i$-rung with end $u_i$, and let $U_3$ be any $3$-rung.
Then $x$ is not a major neighbor of the prism $K_U$ formed by $U_1,
U_2, U_3$.  We can turn $K_W$ into $K_U$ by replacing the rungs one by
one (one at each step).  Along this sequence there are two consecutive
prisms $K$ and $K'$ such that $x$ is a major neighbor of $K$ and not a
major neighbor of $K'$.  Since $K$ and $K'$ are consecutive they
differ by exactly one rung.  Let $K$ be formed by rungs $R_1, R_2,
R_3$, where each $R_i$ has ends $a_i\in A_i$ and $b_i\in B_i$
($i=1,2,3$), and let $A=\{a_1,a_2,a_3\}$ and $B=\{b_1,b_2,b_3\}$; and
let $K'$ be formed by $P_1, R_2, R_3$ for some $i$-rung $P_1$.  Let
$P_1$ have ends $a'_1\in A_1$ and $b'_1\in B_1$, and let $A'=\{a'_1,
a_2, a_3\}$ and $B'=\{b'_1, b_2, b_3\}$.

Let $\alpha=|N(x)\cap A|$, $\beta=|N(x)\cap B|$, $\alpha'=|N(x)\cap
A'|$, $\beta'=|N(x)\cap B'|$.  We know that $\alpha\ge 2$ and
$\beta\ge 2$ since $x$ is a major neighbor of $K$, and
$\min\{\alpha',\beta'\}\le 1$ since $x$ is not a major neighbor of
$K'$.  Moreover, $\alpha'\ge \alpha-1$ and $\beta'\ge \beta-1$ since
$K$ and $K'$ differ by only one rung.  Up to the symmetry on $A,B$,
these conditions imply that the vector $(\alpha,\beta,\alpha',\beta')$
is equal to either $(3,2,3,1)$, $(3,2,2,1)$, $(2,2,2,1)$ or
$(2,2,1,1)$.  In either case we have $\beta=2$ and $\beta'=1$, so $x$
is adjacent to $b_1$, non-adjacent to $b'_1$, and adjacent to exactly
one of $b_2, b_3$, say to $b_3$.  \\
Suppose that $(\alpha',\beta')$ is equal to $(3,1)$ or $(2,1)$.  We
can apply Theorem~\ref{spgt101} to $K'$ and $F=\{x\}$, and it follows
that $x$ satisfies item~4 of that theorem, so $x$ is adjacent to
$a'_1, a_2, b_3$ and has no neighbor in $V(K') \setminus (\{a'_1,
a_2\}\cup V(R_3))$.  But then $V(R_2)\cup\{x,b_3\}$ induces an odd
hole, a contradiction.  So we may assume that
$(\alpha,\beta,\alpha',\beta') = (2,2,1,1)$, which restores the
symmetry between $A$ and $B$.  Since $\alpha=2$ and $\alpha'=1$, $x$
is adjacent to $a_1$, non-adjacent to $a'_1$, and adjacent to exactly
one of $a_2, a_3$.  If $x$ is adjacent to $a_2$, then $K'$ and $\{x\}$
violate Theorem~\ref{spgt101}.  So $x$ is adjacent to $a_3$ and not to
$a_2$, and Theorem~\ref{spgt101} implies that $x$ is a local neighbor
of $K'$ with $N(x)\cap K'\subseteq V(R_3)$, so $x$ has no neighbor in
$P_1$ or $R_2$.  Then we claim that:
\begin{equation}\label{eq:1a}
\longbox{For every $1$-rung $Q_1$, the ends of $Q_1$ are either both
adjacent to $x$ or both non-adjacent to $x$.}
\end{equation}
For suppose the contrary.  Then $x$ is not a major neighbor of the
prism formed by $Q_1, R_2, R_3$, and consequently that prism and the
set $F=\{x\}$ violate Theorem~\ref{spgt101}.  So (\ref{eq:1a}) holds.

\medskip

Let $A'_1=A_1\setminus N(x)$ and $A''_1=A_1\cap N(x)$, and
$B'_1=B_1\setminus N(x)$ and $B''_1=B_1\cap N(x)$.  By (\ref{eq:1a}),
every $1$-rung is either between $A'_1$ and $B'_1$ or between $A''_1$
and $B''_1$.  Let $C'_1$ be the set of vertices of $C_1$ that lie on a
$1$-rung whose ends are in $A'_1\cup B'_1$, and let $C''_1$ be the set
of vertices of $C_1$ that lie on a $1$-rung whose ends are in
$A''_1\cup B''_1$.  By (\ref{eq:1a}), $C'_1$ and $C''_1$ are disjoint
and there is no edge between $A'_1\cup C'_1\cup B'_1$ and $C''_1$ or
between $A''_1\cup C''_1\cup B''_1$ and~$C'_1$.

Pick any $1$-rung $P'_1$ with ends in $A'_1\cup B'_1$.  Then
Theorem~\ref{spgt101} implies (just like for $P_1$) that $x$ is a
local neighbor of the prism formed by $P'_1, R_2, R_3$, so $x$ has no
neighbor in $P'_1$.  It follows that:
\begin{equation}\label{eq:1b}
\mbox{$x$ has no neighbor in $A'_1\cup C'_1\cup B'_1$.}
\end{equation}
Moreover, we claim that:
\begin{equation}\label{eq:1c}
\mbox{$A'_1$ is complete to $A''_1$, and $B'_1$ is complete to
$B''_1$.}
\end{equation}
For suppose on the contrary, up to relabelling vertices and rungs,
that $a'_1$ and $a_1$ are non-adjacent.  Then, by (\ref{eq:1b}),
$V(P_1)\cup\{x, a_1, a_2, b_3\}$ induces an odd hole.  Thus
(\ref{eq:1c}) holds.

\medskip

Let $A'_2=A_2\setminus N(x)$, $A''_2=A_2\cap N(x)$, $B'_2=B_2\setminus
N(x)$ and $B''_2=B_2\cap N(x)$.  Let $C'_2$ be the set of vertices of
$C_2$ that lie on a $2$-rung whose ends are in $A'_2\cup B'_2$, and
let $C''_2$ be the set of vertices of $C_2$ that lie on a $1$-rung
whose ends are in $A''_2\cup B''_2$.  By the same arguments as for the
$1$-rungs, we see that every $2$-rung is either between $A'_2$
and $B'_2$ or between $A''_2$ and $B''_2$, that $C'_2$ and $C''_2$ are
disjoint and that there is no edge between $A'_2\cup C'_2\cup B'_2$
and $C''_2$ or between $A''_2\cup C''_2\cup B''_2$ and $C'_2$.  Also
$x$ has no neighbor in $A'_2\cup C'_2\cup B'_2$, and $A'_2$ is
complete to $A''_2$, and $B'_2$ is complete to $B''_2$.  Note that,
since $xa'_1$ and $xa_2$ are not edges, the sets $A'_1$, $B'_1$,
$C'_1$, $A'_2$, $B'_2$, $C'_2$ are all non-empty.  It follows that the
nine sets
\begin{center}
\begin{tabular}{ccc}
$A'_1$ & $C'_1$ & $B'_1$ \\ 
$A'_2$ & $C'_2$ & $B'_2$ \\ 
$A''_1\cup A''_2\cup A_3$ & $C''_1\cup C''_2\cup C_3\cup\{x\}$ &
$B''_1\cup B''_2\cup B_3$
\end{tabular}
\end{center}
form a hyperprism. So the second item of the theorem holds. 
\end{proof}

\begin{theorem}\label{thm:epr}
Let $G$ be a square-free Berge graph that contains an even prism and
does not contain the line-graph of a bipartite subdivision of $K_4$.
Then $G$ has a good partition.
\end{theorem}
\begin{proof} 
Let $H$ be a maximal even hyperprism in $G$, with partition
$(A_1,\ldots,B_3)$ as in the definition of a hyperprism.  Recall that,
since $H$ is an even hyperprism, $A_i$ is anticomplete to $B_i$ for
each $i$.  Let $M$ be the set of major neighbors of $H$.  Let $Z$ be
the set of vertices of the components of $V(G)\setminus (V(H)\cup M)$
that have no attachment in $H$.  By Lemma~\ref{lem:hyplocal} every
component of $G\setminus (V(H)\cup M\cup Z)$ attaches locally to $H$.
For each $i=1,2,3$, let $F_i$ be the union of the vertex-sets of the
components of $G\setminus (V(H)\cup M\cup Z)$ that attach to $A_i\cup
B_i\cup C_i$.  Let $F_A$ be the union of the vertex-sets of the
components of $G\setminus (V(H)\cup M\cup Z\cup F_1\cup F_2\cup F_3)$
that attach to $A_1\cup A_2\cup A_3$, and define $F_B$ similarly.  By
Lemma~\ref{lem:hyplocal} the sets $F_1$, $F_2$, $F_3$, $F_A$, $F_B$
are well-defined and are pairwise anticomplete to each other, and
$V(G)= V(H)\cup M\cup Z\cup F_1\cup F_2\cup F_3\cup F_A\cup F_B$.

By Theorem~\ref{thm:ehyp}, every vertex in $M$ is complete to at
least two of $A_1, A_2, A_3$ and at least two of $B_1, B_2, B_3$.

Suppose that $M$ contains non-adjacent vertices $x,y$.  By
Theorem~\ref{thm:ehyp}, $x$ and $y$ have a common neighbor $a$ in $A$
and a common neighbor $b$ in $B$.  Then $\{x,y,a,b\}$ induces a
square, a contradiction.  Therefore $M$ is a clique.  By
Lemma~\ref{C4cliques}, $M\cup A_i$ is a clique for at least two values
of $i$, and similarly $M\cup B_j$ is a clique for at least two values
of $j$.  Hence we may assume that both $M\cup A_1$ and $M\cup B_1$ are
cliques.

Define sets $K_1=A_1$, $K_2=M$, $K_3=B_1$, $L=A_2\cup B_2\cup C_2\cup
F_2\cup A_3\cup B_3\cup C_3\cup F_3\cup F_A\cup F_B$ and
$R=V(G)\setminus (K_1\cup K_2\cup K_3\cup L)$.  (So $R=C_1\cup F_1\cup
Z$.)  Every path from $K_3$ to $K_1$ that meets $L$ contains a vertex
from $A_2\cup A_3$, which is complete to $K_1$.  Moreover, since $H$
is an even hyperprism, $K_1$ is anticomplete to $K_3$, and the sets
$C_1, C_2, C_3$ are non-empty, so, picking any vertex $x_i\in C_i$ for
each $i\in\{1,2,3\}$, we see that $\{x_1,x_2,x_3\}$ is a triad with a
vertex in $L$ and a vertex in $R$.  So $(K_1, K_2, K_3, L, R)$ is a
good partition of $V(G)$.
\end{proof}

\section{Odd prisms}

Now we analyze the behavior of major neighbors of an odd hyperprism.
The following theorem is the analogue of Theorem~\ref{thm:ehyp}, but
here we need the assumption that the graph is square-free, and there is an
additional outcome.

Let $H$ be an odd hyperprism in $G$, with partition $(A_1,\ldots,B_3)$
as in the definition of a hyperprism.  For $i=1,2,3$, let
$A^*_i=\{x\in A_i\mid x$ has a neighbor in $B_i\}$ and $B^*_i=\{x\in
B_i\mid x$ has a neighbor in $A_i\}$.  So $A^*_i$ and $B^*_i$ are
either both empty or both non-empty.  Moreover, since $G$ is
square-free, $A^*_i\cup B^*_i$ is non-empty for at most one value of
$i$.

\begin{theorem}\label{thm:ohyp}
Let $G$ be a square-free Berge graph.  Let $H$ be a maximal odd
hyperprism in $G$, with partition $(A_1,\ldots,B_3)$ as in the
definition of a hyperprism, and let $m \in V(G) \setminus V(H)$ be a
major neighbor of $H$.  Then either:
\begin{itemize}
\item 
$m$ is complete to at least two of $A_1, A_2, A_3$ and at least two of
  $B_1, B_2, B_3$, and for every $i \in \{1,2,3\}$ $m$ is complete to
  $A_i^* \cup B_i^*$, or
\item
$A_1^*$ and $B_1^*$ are non-empty, $m$ is complete to $A_1^* \cup
B_1^*$, to at least one of $A_2,A_3$ and to at least one of $B_2,B_3$.
Further, suppose that $m$ has a non-neighbor in at least two of
$B_1,B_2,B_3$.  For $i \in \{1,2,3\}$, let $B_i^1$ be the set of
non-neighbors of $m$ in $B_i$, and $C_i^1$ be the set of all the
vertices of $C_i$ that belong to rungs between $B_i^1$ and $A_i$, and
let $A_i^1$ be the set of all vertices of $A_i$ that are in rungs
whose other end is in $B_i^1$.  Then:
\begin{itemize}
\item  
$A_1^1 \subseteq A_1^*$, and \\
-- for every  path $P$ from $B_1^1 \cup C_1^1$ to $(C_1
\setminus C_1^1) \cup (A_1 \setminus A_1^1)$, some vertex of $V(H)
\setminus (A_1 \cup B_1 \cup C_1)$ has a neighbor in $P^*$. \\ 
-- for every path $P$ from $m$ to $B_1^1\cup C_1^1$ some vertex of
$V(H) \setminus (A_1 \cup B_1 \cup C_1)$ has a neighbor in $P^*$.
\item Let $i \in \{2,3\}$. \\
 -- $A_i^1$ is complete to $A_i \setminus A_i^1$, and $B_i^1$ is
 complete to $B_i \setminus B_i^1$, \\
-- for every  path $P$ from $C_i^1$ to $(A_i \cup B_i \cup C_i)
\setminus (A_i^1 \cup B_i^1 \cup C_i^1)$, some vertex of $V(H)
\setminus (A_i \cup B_i \cup C_i)$ has a neighbor in $P^*$; \\
-- for every  path $P$ from $C_i \setminus C_i^1$ to
$A_i^1\cup B_i^1\cup C_i^1$, some vertex of $V(H) \setminus (A_i \cup
B_i \cup C_i)$ has a neighbor in $P^*$;\\
-- $m$ is complete to $A_i^1$ and for every path $P$ from $m$ to
$B_i^1\cup C_i^1$ some vertex of $V(H) \setminus (A_i \cup B_i \cup
C_i)$ has a neighbor in $P^*$.
\end{itemize}
\end{itemize}
\end{theorem}
\begin{proof}
We first observe that:
\begin{equation}\label{eq:opr2}
\mbox{Every rung of $H$ contains a neighbor of $m$.}
\end{equation}
For suppose on the contrary, up to symmetry, that there is a $1$-rung
$P_1$ that contains no neighbor of $m$.  Let $P_1$ have ends $a_1\in
A_1$ and $b_1\in B_1$.  Suppose $m$ has neighbors $p$ and $q$ such
that $p \in A_2\cup A_3$, $q\in B_2\cup B_3$, and $p$ is non-adjacent
to $q$.  Then $p$-$a_1$-$P_1$-$b_1$-$q$-$m$-$p$ is an odd hole,
contradiction.  Hence, since $m$ is major, we may assume that $m$ is
anticomplete to $A_3\cup B_3$ and has neighbors $a'_1$, $b'_1$, $a_2$,
$b_2$ such that $a'_1\in A_1$, $b'_1\in B_1$, $a_2 \in A_2$, $b_2 \in
B_2$, and $a_2$ and $b_2$ are adjacent.  Note that $A_1\cup A_3$ is
anticomplete to $B_1\cup B_3$ since $A_2^*\cup B_2^*\neq\emptyset$.
Pick any $a_3\in A_3$ and $b_3 \in B_3$.  Suppose that both $a'_1,
b'_1$ have a neighbor in $P_1^*$.  Then there is a $1$-rung $P'_1$
with ends $a'_1, b'_1$ and interior in $P_1^*$, and then
$V(P'_1)\cup\{m\}$ induces an odd hole.  Hence we may assume that
$b'_1$ has no neighbor in $P_1^*$.  Then $b_1b'_1$ is not an edge, for
otherwise $m$-$a_2$-$a_1$-$P_1$-$b_1$-$b'_1$-$m$ is an odd hole.
Suppose that $a'_1$ has a neighbor $c_1$ in $P_1^*$, and choose $c_1$
as close to $b_1$ as possible along $P_1$.  Then
$a'_1$-$c_1$-$P_1$-$b_1$ is a $1$-rung, so it is odd; but then
$m$-$a'_1$-$c_1$-$P_1$-$b_1$-$b_3$-$b'_1$-$m$ is an odd hole.  So
$a'_1$ is also anticomplete to $P_1^*$, and by symmetry $a_1a'_1$ is
not an edge.  But then
$m$-$a'_1$-$a_3$-$a_1$-$P_1$-$b_1$-$b_3$-$b'_1$-$m$ is an odd hole.
This proves (\ref{eq:opr2}).

\medskip

\begin{equation}\label{mabstar}
\mbox{For each $i$, $m$ is complete to $A^*_i\cup B^*_i$.}
\end{equation}
For suppose the contrary.  Then there are vertices $u_i\in A^*_i$ and
$v_i\in B^*_i$ such that $u_iv_i$ is an edge and $m$ has a
non-neighbor in $\{u_i,v_i\}$.  Since $m$ is a major neighbor of $H$,
it has a neighbor $a$ in $(A_1\cup A_2\cup A_3)\setminus A_i$ and a
neighbor $b$ in $(B_1\cup B_2\cup B_3)\setminus B_i$.  Then the
subgraph induced by $\{m,a,b,u_i,v_i\}$ contains a square or a $C_5$,
a contradiction.  Thus (\ref{mabstar}) holds.

\medskip

In view of (\ref{mabstar}) we may assume that $m$ does not satisfy the
property of being complete to at least two of $B_1, B_2, B_3$ (for
otherwise the first outcome of the theorem holds).  So we may assume
that $m$ is not complete to $B_1$, not complete to $B_2$, and
(possibly exchanging the roles of $B_2$ and $B_3$), that $m$ has a
neighbor in $B_3$.  Let $b_2\in B_2$ be a non-neighbor of $m$ and
$b_3\in B_3$ be a neighbor of~$m$.
\begin{equation}\label{eq:opr3}
\longbox{Let $b_1$ be a non-neighbor of $m$ in $B_1$, and let
$a_1$-$P_1$-$b_1$ be a rung through $b_1$.  Then $m$ is adjacent to
$a_1$ and anticomplete to $V(P_1) \setminus \{a_1\}$.  Moreover, let
$P$ be a path from $m$ to $V(P_1) \setminus \{a_1\}$.  Then some
vertex of $V(H) \setminus (A_1 \cup B_1 \cup C_1)$ has a neighbor in
$P^*$.}
\end{equation}
Let $a_2$-$P_2$-$b_2$ be a rung through $b_2$.  If there is a path $P$
violating \eqref{eq:opr3}, then by \eqref{eq:opr2} $m$ has a neighbor in
$P_2$, and therefore we can link $m$ to $\{b_1, b_2, b_3\}$
via $P$, $P_2$ and $m$-$b_3$, a contradiction to Lemma~\ref{lem:link}.
Thus no such path exists, and in particular $m$ is anticomplete to
$V(P_1) \setminus \{a_1\}$.  Now it follows from (\ref{eq:opr2}) that
$m$ is adjacent to $a_1$.  This proves (\ref{eq:opr3}).  Note that the
analogue of (\ref{eq:opr3}) also holds for $B_2$.

\begin{equation}\label{eq:opr4}
\longbox{For every symbol $X$ in $\{A,B,C\}$ there is a partition of
$X_1$ into two sets $X'_1$ and $X''_1$ such that: \\
-- $A'_1$ is complete to $A''_1$, and $B'_1$ is complete to $B''_1$;
\\
-- $A'_1$ is anticomplete to $B'_1$; \\
-- for every path $P$ from $C'_1$ to $A''_1\cup B''_1\cup C''_1$, some
vertex of $V(H) \setminus (A_1 \cup B_1 \cup C_1)$ has a neighbor in
$P^*$, and in particular $C_1'$ is anticomplete to $A_1''
\cup B_1'' \cup C_1''$;\\
-- for every path $P$ from $C''_1$ to $A'_1\cup B'_1\cup C'_1$, some
vertex of $V(H) \setminus (A_1 \cup B_1 \cup C_1)$ has a neighbor in
$P^*$, and in particular $C_1''$ is anticomplete to $A_1' \cup B_1'
\cup C_1'$; \\
-- $m$ is complete to $A'_1$ and for every path $P$ from $m$ to
$B'_1\cup C'_1$ some vertex of $V(H) \setminus (A_1 \cup B_1 \cup
C_1)$ has a neighbor in~$P^*$, and in particular $m$ is anticomplete
to $B_1' \cup C_1'$.}
\end{equation}

Pick rungs $a_2$-$P_2$-$b_2$ and $a_3$-$P_3$-$b_3$ containing $b_2$
and $b_3$ respectively.  By (\ref{eq:opr3}), $m$ is adjacent to $a_2$.

Let $B'_1=\{y\in B_1\setminus B^*_1 \mid y$ is non-adjacent to $m$ and
there exists a rung from $A_1\setminus A^*_1$ to $y\}$, and let
$A'_1=\{x\in A_1\setminus A^*_1\mid$ there is a rung from $x$ to
$B'_1\}$.  Let $C'_1=\{z\in C_1\mid z$ lies on a rung between $B'_1$
and $A'_1\}$.  So $m$ is anticomplete to $B'_1$ and, by
(\ref{eq:opr3}), $m$ is complete to $A'_1$ and anticomplete to $C'_1$.
Let $B''_1=B_1\setminus B'_1$, $A''_1=A_1\setminus A'_1$, and
$C''_1=C_1\setminus C'_1$.  Let $Q$ be any rung with ends $x\in A'_1$
and $y\in B'_1$.  We prove five  claims (a)--(e) as follows.

\medskip

(a) {\it For every path $P$ from $B''_1$ to $A'_1\cup C'_1$ some
vertex of $V(H) \setminus (A_1 \cup B_1 \cup C_1)$ has a neighbor in
$P^*$.} \\
We know that $B''_1$ is anticomplete to $A'_1$ since $A'_1\subseteq
A_1\setminus A^*_1$.  Suppose up, to relabelling, that there is a path
$P$ from some vertex $b_1$ in $B''_1$ to a vertex of
$Q\setminus\{y\}$, contradicting (a).  Then there is a  path
$Q'$ from $x$ to $b_1$ with interior in $P^* \cup Q^*$.  By the
maximality of $H$, ${Q'}^* \subseteq C_1$, and $Q'$  has odd length at
least $3$. By the definition of $B'_1$ and the
existence of $Q'$ imply that $b_1$ is adjacent to $m$. Suppose $m$ has a
neighbor in $P^*$.  Since no vertex of
$V(H) \setminus (A_1 \cup B_1 \cup C_1)$ has a neighbor in $P^*$,
it follows from the second statement of \eqref{eq:opr3} that
$P$ is a path from $b_1$ to $a_1$, and $P^* \cap V(Q)=\emptyset$,
and that $P^*$ is anticomplete to $V(Q)$. If $y$ is adjacent to $b_1$, then
$x-P-b_1-y-Q-x$ is an odd hole, a contradiction. Thus $b_1$ is non-adjacent to
$y$, and so $b_1-m-x-Q-y-b_2-b_1$, is an odd hole, again a contradiction.
This proves that $m$ is anticomplete
to $P^*$. Recall that $m$ is adjacent to $x$
and anticomplete to $V(Q) \setminus \{x\}$.
It follows that $m$ is
anticomplete to ${Q'}^*$.  consequently 
$V(Q')\cup\{m\}$ induces an odd hole.  Since this holds for all $Q$,
claim (a) is established.

\medskip

(b) {\it $B''_1$ is complete to $B'_1$.} \\ 
Suppose, up to relabelling, that some $b_1$ in $B''_1$ is not adjacent
to $y$.  Recall that $m$ is anticomplete to $V(Q) \setminus \{x\}$.
By (a) $b_1$ has no neighbor in $Q$.  Then $b_1$ is
non-adjacent to $m$, for otherwise $x$-$Q$-$y$-$b_2$-$b_1$-$m$-$x$
induces an odd hole.  Pick a rung $a_1$-$P_1$-$b_1$.  By
(\ref{mabstar}), $b_1\notin B^*_1$, hence, by the definition of
$B'_1$, we have $a_1\in A^*_1$, and so $a_1$ has a neighbor $b^*_1\in
B^*_1$.  If $b_1^*$ is not adjacent to $y$, then, by the same argument
as for $b_1$ it follows that $b_1^*$ is not adjacent to $m$, which
contradicts (\ref{mabstar}).  Therefore $b^*_1$ is adjacent to $y$
and, by (\ref{mabstar}), to $m$.  By (a) $b^*_1$ has no neighbor in
$Q\setminus y$.  We know that $a_1$ is not adjacent to $y$ since
$y\notin B^*_1$.  Moreover $a_1$ has no neighbor in $Q\setminus x$,
for otherwise we can link $a_1$ to $\{b_3, y, b^*_1\}$ via
$a_3$-$P_3$-$b_3$, $Q\setminus x$ and $a_1$-$b^*_1$, a contradiction
to Lemma~\ref{lem:link}.  Then $xa_1$ is an edge, for otherwise
$x$-$Q$-$y$-$b^*_1$-$a_1$-$a_3$-$x$ is an odd hole.  There is no edge
between $Q$ and $P_1$ except $a_1x$, for otherwise there would be a
rung from $x$ to $b_1$, implying $b_1 \in B'_1$.  But then
$b_1$-$P_1$-$a_1$-$x$-$Q$-$y$-$b_3$-$b_1$ is an odd hole.  Thus
$B''_1$ is complete to $y$, and since this holds for all $Q$, the
claim is established.

\medskip


 
(c) {\it For every path $P$ from $A''_1$ to $B'_1$, some vertex of
$V(H) \setminus (A_1 \cup B_1 \cup C_1)$ has a neighbor in $P^*$.} \\
For suppose up to relabelling that $a_1$-$P_1$-$y$ is a path
contradicting (c).  By the maximality of $H$, $P_1$ is a rung of $H$.
By the definition of $B'_1$ we have $a_1\in A^*_1$, so $a_1$ has a
neighbor $b_1^*\in B^*_1$.  By (a) and (b), $b^*_1$ is adjacent to $y$
and anticomplete to $Q\setminus y$. Since $y \not \in B_1^*$, it follows that
$a_1$ is not adjacent to $y$. Then $a_1$ has no neighbor in
$Q\setminus x$, for otherwise we can link $a_1$ to $\{y,b^*_1,b_2\}$
via $Q\setminus x$, $a_1$-$b^*_1$ and $a_2$-$P_2$-$b_2$.  Moreover
$a_1$ is adjacent to $x$, for otherwise
$a_1$-$b^*_1$-$y$-$Q$-$x$-$a_2$-$a_1$ is an odd hole.  Since $V(Q)\cup
V(P_1)$ does not induce an odd hole, there is an edge between
$Q\setminus y$ and $P_1\setminus\{a_1,y\}$.  Since
$b_1^*$-$y$-$P_1$-$a_1$-$b_1^*$ cannot be an odd hole, $b_1^*$ has a
neighbor in $P_1\setminus\{a_1,y\}$.  But this implies the existence
of a rung between $x$ and $b^*_1$, which contradicts (a).

\medskip

(d) {\it $A''_1$ is complete to $A'_1$.} \\
Suppose on the contrary that some $a_1$ in $A''_1$ is non-adjacent to
$x$.  Let $a_1$-$P_1$-$b_1$ be a rung.  By (c), $b_1\in B''_1$, and by
(a)--(c), the only edge between $P_1$ and $Q$ is $b_1y$.  Then
$a_1$-$P_1$-$b_1$-$y$-$Q$-$x$-$a_3$-$a_1$ is an odd hole, a
contradiction.

\medskip

(e) {\it For every path $P$ from $C'_1$ to $A''_1\cup C''_1\cup
B''_1$, some vertex of $V(H) \setminus (A_1 \cup B_1 \cup C_1)$ has a
neighbor in $P^*$, and for every path $P$ from $C''_1$ to $A'_1\cup
C'_1\cup B'_1$ some vertex of $V(H) \setminus (A_1 \cup B_1 \cup C_1)$
has a neighbor in $P^*$.} \\
Indeed, in the opposite case there is a path that violates (a) or (c).

\medskip

It follows from \eqref{eq:opr3} and claims (a)--(e) that all the properties described in
(\ref{eq:opr4}) are satisfied.  So (\ref{eq:opr4}) holds.

\medskip

By (\ref{eq:opr4}), if $B_1' \neq \emptyset$, then $(A_1, C_1, B_1)$
is a strip 
such that $m$ is complete to $A'_1$ and anticomplete to
$B'_1\cup C'_1$.  Likewise, if $B_2' \neq \emptyset$, then
$(A'_2, C'_2, B'_2)$ is a strip such that $m$ is complete to $A'_2$ and
anticomplete to $B'_2\cup C'_2$.  So if both $B_1' \neq \emptyset$ and
$B_2' \neq \emptyset$, using the properties described in
(\ref{eq:opr4}), we obtain a hyperprism:
\begin{center}\begin{tabular}{ccc}
$A'_1$ & $C'_1$ & $B'_1$\\
$A'_2$ & $C'_2$ & $B'_2$ \\
$A_3\cup A''_1\cup A''_2\cup \{m\}$ & $C_3\cup C''_1\cup C''_2$ &
$B_3\cup B''_1\cup B''_2$.
\end{tabular}\end{center}
contrary to the maximality of $H$.

Thus we may assume that $B_1' = \emptyset$, and consequently
$A_1^*,B_1^*$ are both non-empty.  We claim that the following holds:

\medskip

\begin{equation}\label{eq:new}
\longbox{Let $B_1^1$ be the set of non-neighbors of $m$ in $B_1$.
Then every rung with an end in $B_1^1$ has its other end in $A_1^*$.
Moreover, let $C_1^1$ be the set of all the vertices of $C_1$ that
belong to rungs between $B_1^1$ and $A_1^*$.  Then for every path $P$
from $B_1^1 \cup C_1^1$ to $(C_1 \setminus C_1^1) \cup (A_1 \setminus
A_1^*)$ some vertex of $V(H) \setminus (A_1 \cup B_1 \cup C_1)$ has a
neighbor in $P^*$.}
\end{equation}

Since $B_1'=\emptyset$, it follows that every rung of $H$ with an end
in $B_1^1$ has its other end in $A_1^*$.  Next suppose that there is a
path $P$ from $c \in B_1^1 \cup C_1^1$ to
$d \in (C_1 \setminus C_1^1) \cup (A_1 \setminus A_1^*)$ violating
\eqref{eq:new}.  By the
definition of a hyperprism, there is a  path $P_c$ from $c$
to $b \in B_1^1$ with interior in $C_1$, and a  path $P_d$
from $d$ to $a \in A_1 \setminus A_1^*$. Since $P^*$ violates \eqref{eq:new},
no vertex of  $V(H) \setminus (A_1 \cup B_1 \cup C_1)$ has a
neighbor in $P^*$, and therefore $P^* \cap (A_1 \cup B_1)=\emptyset$.
Since $B_1^1$ is
anticomplete to $A_1$, there is a  path $Q$ from $a$ to $b$ with
non-empty interior contained in
$V(P) \cup V(P_a) \cup V(P_b) \setminus \{a,b\}$.
By the maximality of $H$,
$Q$  is a rung of $H$, contrary
to the fact that every rung with an end in $B_1^1$ has its other end
in $A_1^*$.  This proves (\ref{eq:new}).

\medskip

Since $A_1^*,B_1^*$ are non-empty, it follows that $B_2' \neq \emptyset$,
and so $(A_2',C_2',B_2')$ is a strip.
By (\ref{mabstar})  $m$ has a
neighbor in $B_1$. Suppose that  $m$ is not complete to $B_3$. Then there is
symmetry between $B_1$ and $B_3$, and therefore $B_3' \neq \emptyset$
and $(A_3',C_3',B_3')$ is a strip. Applying (\ref{eq:opr4}) to
$A_2,B_2, C_2$ and to  $A_3,B_3,C_3$,
we obtain a hyperprism:
\begin{center}\begin{tabular}{ccc}
$A'_2$ & $C'_2$ & $B'_2$\\
$A'_3$ & $C'_3$ & $B'_3$ \\
$A_1\cup A''_2\cup A''_3\cup \{m\}$ & $C_1\cup C''_2\cup C''_3$ &
$B_1\cup B''_2\cup B''_3$.
\end{tabular}\end{center}
contrary to the maximality of $H$.

Thus we may assume that $m$ is complete to $B_3$.  Since $m$ has a
neighbor in $A_1^* \subseteq A_1$, reversing the roles of $A$ and $B$
we may assume that $m$ is complete to at least one of $A_2,A_3$.  Now
by \eqref{eq:opr3}, \eqref{eq:opr4} and \eqref{eq:new}, the second outcome of
the theorem holds.
\end{proof}

\begin{theorem}\label{thm:opr}
Let $G$ be a square-free Berge graph that contains an odd prism and
does not contain the line-graph of a bipartite subdivision of $K_4$.
Then $G$ admits a good partition.
\end{theorem}
\begin{proof} 
Since $G$ contains an odd prism, it contains a maximal odd hyperprism
$(A_1, C_1, B_1, A_2, C_2, B_2, A_3, C_3, B_3)$ which we call $H$.
Let $M$ be the set of major neighbors of $H$.  Let $M_{good}$ be the
set of all the vertices of $M$ that are complete to at least two of
$A_1,A_2,A_3$ and to at least two of $B_1,B_2,B_3$, and let $M_{bad}$
be the remaining major neighbors of $H$.  Let $Z$ be the set of
vertices of the components of $V(G)\setminus (V(H)\cup M)$ that have
no attachment in $H$.  Since $H$ is maximal, by
Lemma~\ref{lem:hyplocal} there is a partition of $V(G)\setminus
(V(H)\cup M\cup Z)$ into sets $F_1$, $F_2$, $F_3$, $F_A$, $F_B$ such
that:
\begin{itemize}
\item
For $i=1,2,3$, $N(F_i)\subseteq A_i\cup C_i\cup B_i\cup M$;
\item
$N(F_A)\subseteq A_1\cup A_2\cup A_3\cup M$ and $N(F_B)\subseteq
B_1\cup B_2\cup B_3\cup M$;
\item
The sets $Z, F_1, F_2, F_3, F_A, F_B$ are pairwise anticomplete to
each other.
\end{itemize}

We observe that:
\begin{equation}\label{eq:opr1}
\longbox{At least two of $A_1,A_2,A_3$ are cliques, and at least two
of $B_1,B_2,B_3$ are cliques.}
\end{equation}
This follows directly from Lemma~\ref{C4cliques} (with $K=\emptyset$).

\medskip

Since $H$ is maximal, \ref{thm:ohyp} implies that:
\begin{equation}\label{eq:opr5}
\longbox{Let $m \in M$.  For every $i \in \{1,2,3\}$, $m$ is complete
to $A_i^* \cup B_i^*$.  Moreover, if $i,j \in \{1,2,3\}$ are distinct
and $A_i^*=A_j^*=\emptyset$, then $m$ is complete to at least one of
$A_i,A_j$ and to at least one of $B_i,B_j$.}
\end{equation}

We claim that:
\begin{equation}\label{eq:opr6}
\mbox{$M$ is a clique.}
\end{equation}
Suppose that there are non-adjacent vertices $m_1, m_2$ in $M$.  By
(\ref{eq:opr5}), $m_1$ and $m_2$ have a common neighbor in $A_1\cup
A_2\cup A_3$.  Therefore let $a_1$ be a common neighbor of $m_1$ and
$m_2$ in $A_1$.  If $m_1$ and $m_2$ have a common neighbor $b \in
B_2\cup B_3$, then $\{m_1,m_2,a_1,b\}$ induces a square, a
contradiction.  In view of \eqref{eq:opr5}, $A_2^*=A_3^*=\emptyset$,
and we may assume up to symmetry that $m_1$ is not complete to $B_2$,
so it is complete to $B_3$, and consequently that $m_2$ is not
complete to $B_3$, and so it is complete to $B_2$.  Pick a
non-neighbor $b_2$ of $m_1$ in $B_2$ and a non-neighbor $b_3$ of $m_2$
in $B_3$.  Then $\{m_1, m_2, a_1, b_2, b_3\}$ induces a $C_5$, a
contradiction.  This proves~(\ref{eq:opr6}).
\begin{equation}\label{eq:opr7}
\longbox{$M_{good}\cup A_i$ is a clique for at least two values of
$i$, and similarly $M_{good} \cup B_j$ is a clique for at least two
values of $j$.}
\end{equation}
This follows directly from (\ref{eq:opr5}), (\ref{eq:opr6}) and
Lemma~\ref{C4cliques}.  Thus (\ref{eq:opr7}) holds.

\begin{equation}\label{eq:opr7new}
\longbox{Let $j,k \in \{1,2,3\}$ be distinct and such that
$A_j^*=A_k^*=\emptyset$.  Then either $M \cup A_j$ or $M \cup A_k$ is
a clique, and similarly either $M\cup B_j$ or $M \cup B_k$ is a
clique.}
\end{equation}
This follows directly from (\ref{eq:opr5}), (\ref{eq:opr6}) and
Lemma~\ref{C4cliques} applied to $M,A_j,A_k$ and to $M,B_j,B_k$.
Thus (\ref{eq:opr7new}) holds.

\medskip

Since $G$ is square-free, we may assume, up to symmetry, that $A_1$ is
anticomplete to $B_1$, and that $A_2$ is anticomplete to $B_2$, and so
$C_1\neq\emptyset$ and $C_2\neq\emptyset$.  Pick any $x_1\in C_1$,
$x_2\in C_2$ and $a_3\in A_3$.  So $\{x_1, x_2, a_3\}$ is a triad
$\tau$.

Suppose that both $M \cup A_1$ and $M \cup B_1$ are cliques.  Set
$K_1=A_1$, $K_2=M$, $K_3=B_1$, $L=A_2\cup B_2\cup C_2\cup F_2\cup
A_3\cup B_3\cup C_3\cup F_3\cup F_A\cup F_B$ and $R=V(G)\setminus
(K_1\cup K_2\cup K_3\cup L)$.  (So $R=C_1\cup F_1\cup Z$.)  We observe
that $K_1$ is anticomplete to $K_3$, every path from $K_3$ to $K_1$
with interior in $L$ contains a vertex from $A_2\cup A_3$, which is
complete to $K_1$, and $\tau$ is a triad with a vertex in $L$ and a
vertex in $R$.  Thus $(K_1, K_2, K_3, L, R)$ is a good partition of
$V(G)$.  The same holds if both $M \cup A_2$ and $M \cup B_2$ are
cliques.

By \eqref{eq:opr7new}, we may assume that $A_3^* \neq \emptyset$, $M
\cup A_1$ and $M \cup B_2$ are cliques, and neither of $M \cup A_2$
and $M \cup B_1$ is a clique.  Suppose that $M_{bad}=\emptyset$.  Then
by \eqref{eq:opr7} $M \cup A_3$ and $M \cup B_3$ are both cliques.
Set $K_1=A_1$, $K_2=M$, $K_3=B_2 \cup B_3$, $L=A_2\cup C_2\cup F_2\cup
A_3 \cup C_3\cup F_3\cup F_A$ and $R=V(G)\setminus (K_1\cup K_2\cup
K_3\cup L)$.  (So $R=C_1\cup B_1 \cup F_1\cup Z$.)  We observe that
$K_1$ is anticomplete to $K_3$, every path from $K_3$ to $K_1$ with
interior in $L$ contains a vertex from $A_2\cup A_3$, which is
complete to $K_1$, and $\tau$ is a triad with a vertex in $L$ and a
vertex in $R$.  Thus $(K_1, K_2, K_3, L, R)$ is a good partition of
$V(G)$.  Thus we may assume that $M_{bad}\neq\emptyset$.  Let $M_A$ be
the set of vertices of $M$ with a non-neighbor in both $A_2$ and
$A_3$, and let $M_B$ be the set of vertices of $M$ with a non-neighbor
in both $B_1$ and $B_3$.  So $M_{bad}=M_A\cup M_B$.  By
Lemma~\ref{C4cliques} applied with $K=M \setminus M_B$, $X_1=B_1$ and
$X_2=B_3$, we deduce that either $(M \setminus M_B) \cup B_3$ is a
clique, or $(M \setminus M_B) \cup B_1$ is a clique.

Suppose first that $(M \setminus M_B) \cup B_1$ is a clique.  Then,
since $M \cup B_1$ is not a clique, it follows that $M_B \neq
\emptyset$.  Let $m \in M_B$ be chosen with $N(m) \cap B_1$ maximal;
let $B_1'=B_1 \cap N(m)$ and $B_1'' = B_1 \setminus B_1'$.  Since
$B_1,M$ are both cliques and $G$ has no $C_4$, it follows from the
choice of $m$ that $M_B$ is anticomplete to $B_1''$.  Let $C_1''$ be
the set of vertices of $C_1$ that are in rungs with vertices of
$B_1''$, let $A_1''$ be the set of vertices in $A_1$ that are in rungs
with vertices of $B_1''$, and let $F_1''$ be the set of vertices of
$F_1$ such that there is a path from them to $B_1'' \cup C_1''$ with
interior in $F_1''$.  Recall that every vertex of $M_B$ has a
non-neighbor in $B_1$ and a non-neighbor in $B_3$, and so the second
outcome of Theorem~\ref{thm:ohyp} holds.  It follows that the set
$C_1'' \cup F_1''$ is anticomplete to $M_B \cup (C_1 \setminus C_1'')
\cup (F_1 \setminus F_1'')$, and $M_B$ is complete to $A_1''$.  Now
set $K_1=A_1''$, $K_2=M \setminus M_B$, $K_3=B_1''$, $R=C_1'' \cup
F_1''$ and $L=V(G) \setminus (K_1 \cup K_2 \cup K_3 \cup R)$.  By
Theorem~\ref{thm:ohyp} the set $L$ is anticomplete to $R$.  We know
that $K_1\cup K_2$ is a clique (because $M\cup A_1$ is a clique) and
$K_2\cup K_3$ is a clique (because of the current assumption that $(M
\setminus M_B) \cup B_1$ is a clique).  Moreover, $K_1$ is
anticomplete to $K_3$, and, again by Theorem~\ref{thm:ohyp}, every
path from $K_3$ to $K_1$ with interior in $L$ contains a vertex of
$A_2 \cup A_3 \cup (A_1 \setminus A_1'') \cup M_B$, which is complete
to $A_1''$.  We may assume that $\tau \cap C_1 \subseteq C_1''$, and
thus $\tau$ is a triad with a vertex in $L$ and a vertex in $R$.
Consequently, $(K_1,K_2,K_3,L,R)$ is a good partition.

\medskip

Hence we may assume that $(M \setminus M_B) \cup B_1$ is not a clique,
and so $(M \setminus M_B) \cup B_3$ is a clique.  By symmetry we may
assume that $(M \setminus M_A) \cup A_3$ is a clique.

Switching the roles of $A$ and $B$ if necessary, we may assume that
$M_B\neq\emptyset$.  Let $B_1''$ be the set of vertices in
$B_1$ that are not complete to $M_B$, and let $B_1' = B_1 \setminus
B_1''$.  So $B_1''\neq\emptyset$.  Let $C_1''$ be the set of vertices
of $C_1$ that are in rungs with vertices of $B_1''$, let $A_1''$ be
the set of vertices in $A_1$ that are in rungs with vertices of
$B_1''$, and let $F_1''$ be the set of vertices of $F_1$ such that
there is a path from them to $B_1'' \cup C_1''$ with interior in
$F_1''$.  By \ref{thm:ohyp} the set $C_1'' \cup F_1''$ is anticomplete
to $(C_1 \setminus C_1'') \cup (F_1 \setminus F_1'')$; recall that $M
\cup A_1$ is a clique.  Let $A_3''$ be the set of vertices of $A_3$
that are not complete to $M_A$, and let $A_3' =A_3 \setminus A_3''$.
Let $C_3''$ be the set of vertices of $C_3$ in rungs with $A_3''$, and
$F_3''$ the set of vertices of $F_3$ such that there is a path to them
from $A_3'' \cup C_3''$ with interior in $F_3''$.  Let $C_3'=C_3
\setminus C_3''$ and $F_3'=F_3 \setminus F_3''$.

We claim that $M \setminus M_B$ is complete to $B_1'$.  Suppose $m'
\in M \setminus M_B$ has a non-neighbor in $b_1' \in B_1'$.  Choose $m
\in M_B$.  Since $m \in M_B$, some $b_3 \in B_3$ is non-adjacent to
$m$.  Now $\{m,m',b_1',b_3\}$ induces a $C_4$, a contradiction.  

Let $K_1=A_1'' \cup A_3'$, $K_2=M$, $K_3=B_1' \cup B_2 \cup B_3^*$,
$R=B_1'' \cup C_1'' \cup F_1'' \cup F_B \cup (B_3 \setminus B_3^*)
\cup C_3' \cup F_3'$ and $L=V(G) \setminus (K_1 \cup K_2 \cup K_3 \cup
R)$.  By Theorem~\ref{thm:ohyp} the set $L$ is anticomplete to $R$,
and every rung of $H$ with an end in $A_3''$ has its other end in
$B_3^*$.  Since $M \cup A_1$ is a clique, $A_3 \cup (M \setminus M_A)$
is a clique, and since by the definition of $A_3'$, $M_A$ is complete
to $A_3'$, it follows that $K_1 \cup K_2$ is a clique.  Since $M \cup
B_2$ is a clique, $M \setminus M_B$ is complete to $B_1'$, by the
definition of $B_1'$, $M_B$ is complete to $B_1'$, $B_3 \cup (M
\setminus M_B)$ is a clique, and by Theorem~\ref{thm:ohyp} $M_B$ is
complete to $B_3^*$, it follows that $K_2 \cup K_3$ is a clique.

We now check that condition (iii) in the definition of a good
partition holds.  Suppose that $P$ is a path from $K_3$ to $K_1$ with
$P^* \neq \emptyset$ and $P^* \subseteq L$.  We may assume that $V(P)$
is disjoint from $A_2 \cup A_3'' \cup (A_1 \setminus A_1'')$, which is
complete to $A_1'' \cup A_3'$ (recall that $A_1$ and $A_3$ are both
cliques).  It follows from Theorem~\ref{thm:ohyp} that the ends of $P$
are in $A_3'$ and in $B_3^*$, and $P^* \subseteq C_3'' \cup F_3''$.
In particular, $A_3'' \neq \emptyset$.  We claim that some $m \in M_A$
is anticomplete to $A_3''$.  Choose $m \in M_A$ with $N(m) \cap A_3''$
minimal.  We may assume that $A_3'' \cap N(m) \neq \emptyset$; let $a
\in A_3''$ be a neighbor of $m$.  It follows from the definition of
$A_3''$ that some $m' \in M_A$ is non-adjacent to $a$.  By the choice
of $m$, there is $a' \in A_3''$ adjacent to $m'$ and not to $m$.  But
now since $A_3$ and $M_A$ are both cliques, $\{m,m',a,a'\}$ induces a
$C_4$, a contradiction.  This proves the claim; let $m \in M_A$ be
anticomplete to $A_3''$.  By Theorem~\ref{thm:ohyp} $m$ is
anticomplete to $C_3'' \cup F_3''$, and so $m$ is anticomplete to
$P^*$.  But now $V(P) \cup \{m\}$ induces an odd hole, a
contradiction.  This proves that condition (iii) holds.

Next we check that condition (iv) in the definition of a good
partition is satisfied.  Suppose that some $l \in L$ has a neighbor
$k_1 \in K_1$ and a neighbor $k_3 \in K_3$.  Then $l \in C_3'' \cup
F_3''$.  Since $H$ is an odd prism, $k_1$-$l$-$k_3$ is not a rung of
length two, and therefore $k_1$ is adjacent to $k_3$.  This proves
that condition (iv) holds.  Finally, we may assume that $\tau \cap C_1
\subseteq C_1''$, and thus $\tau$ is a triad with a vertex in $L$ and
a vertex in $R$.  This proves that $(K_1,K_2,K_3,L,R)$ is a good
partition, and completes the proof of the theorem.
\end{proof}

\section{Line-graphs}

The goal of this section will be to prove the following decomposition
theorem.

\begin{theorem}\label{thm:lgb}
Let $G$ be a square-free Berge graph, and assume that $G$ contains the
line-graph of a bipartite subdivision of $K_4$.  Then $G$ admits a
good partition.
\end{theorem}

Before proving this theorem, we need some definitions from
\cite{CRST}.

In a graph $H$, a \emph{branch} is a path whose interior vertices have
degree~$2$ and whose ends have degree at least~$3$.  A
\emph{branch-vertex} is any vertex of degree at least~$3$.   

In a graph $G$, an \emph{appearance} of a graph $J$ is any induced
subgraph of $G$ that is isomorphic to the line-graph $L(H)$ of a
bipartite subdivision $H$ of $J$.  An appearance of $J$ is
\emph{degenerate} if either (a) $J=H=K_{3,3}$ (the complete bipartite
graph with three vertices on each side) or (b) $J=K_4$ and the four
vertices of $J$ form a square in $H$.  Note that a degenerate
appearance of a graph contains a square since in either case (a) or
(b) the graph $H$ contains a square.  An appearance $L(H)$ of $J$ in
$G$ is \emph{overshadowed} if there is a branch $B$ of $H$, of length
at least $3$, with ends $b_1,b_2$, such that some vertex of $G$ is
non-adjacent in $G$ to at most one vertex in $\delta_H(b_1)$ and at
most one in $\delta_H(b_2)$, where $\delta_H(b)$ denotes the set of
edges of $H$ (vertices of $L(H)$) of which $b$ is an end.

An \emph{enlargement} of a $3$-connected graph $J$ (also called a
\emph{$J$-enlargement}) is any $3$-connected graph $J'$ such that
there is a proper induced subgraph of $J'$ that is isomorphic to a
subdivision of $J$.

To obtain a decomposition theorem for graphs containing line graphs of
bipartite graphs, we first thicken the line graph into an object
called a strip system, and then study how the components of the rest
of the graph attach to the strip system.

\medskip

Let $J$ be a $3$-connected graph, and let $G$ be a Berge graph.  A
\emph{$J$-strip system $(S,N)$ in $G$} means
\begin{itemize}
\item 
for each edge $uv$ of $J$, a subset $S_{uv}=S_{vu}$ of $V(G)$,
\item 
for each vertex $v$ of $J$, a subset $N_v$ of $V(G)$,
\item
$N_{uv} = S_{uv} \cap N_u$, 
\end{itemize}
satisfying the following conditions (where for $uv \in E(J)$, a
\emph{$uv$-rung} means a path $R$ of $G$ with ends $s, t$, say, where
$V(R) \subseteq S_{uv}$, and $s$ is the unique vertex of $R$ in $N_u$,
and $t$ is the unique vertex of $R$ in $N_v$):
\begin{itemize}
\item 
The sets $S_{uv}$ ($uv \in E(J)$) are pairwise disjoint;
\item
For each $u\in V(J)$, $N_u\subseteq \bigcup_{uv\in E(J)}S_{uv}$;
\item 
For each $uv \in E(J)$, every vertex of $S_{uv}$ is in a $uv$-rung;
\item 
For any two edges $uv,wx$ of $J$ with $u,v,w,x$ all distinct, there
are no edges between $S_{uv}$ and $S_{wx}$;
\item 
If $uv$, $uw$ in $E(J)$ with $v \neq w$, then $N_{uv}$ is complete to
$N_{uw}$ and there are no other edges between $S_{uv}$ and $S_{uw}$;
\item 
For each $uv\in E(J)$ there is a special $uv$-rung such that for every
cycle $C$ of $J$, the sum of the lengths of the special $uv$-rungs for
$uv\in E(C)$ has the same parity as $|V(C)|$.
\end{itemize}
The vertex set of $(S,N)$, denoted by $V(S,N)$, is the set
$\bigcup_{uv\in E(J)} S_{uv}$.

Note that $N_{uv}$ is in general different from $N_{vu}$.  On the
other hand, $S_{uv}$ and $S_{vu}$ mean the same thing.

The following two properties follow from the definition of a strip
system:
\begin{itemize}
\item
For distinct $u,v\in V(J)$, we have $N_u\cap N_v\subseteq S_{uv}$ if
$uv\in E(J)$, and $N_u\cap N_v=\emptyset$ if $uv\not\in E(J)$.
\item
For $uv\in E(J)$ and $w\in V(J)$, if $w\neq u,v$ then $S_{uv}\cap
N_w=\emptyset$.
\end{itemize}
In 8.1 from \cite{CRST} it is shown that for every $uv\in E(J)$, all
$uv$-rungs have lengths of the same parity.  It follows that the final
axiom is equivalent to:
\begin{itemize}
\item
For every cycle $C$ of $J$ and every choice of $uv$-rung for every
$uv\in E(C)$, the sums of the lengths of the $uv$-rungs has the same
parity as $|V(C)|$.  In particular, for each edge $uv\in E(J)$, all
$uv$-rungs have the same parity.
\end{itemize}

By a {\em choice of rungs} we mean the choice of one $uv$-rung for
each edge $uv$ of $J$.  It follows from the preceding points that for
every choice of rungs the subgraph of $G$ induced by their union is
the line-graph of a bipartite subdivision of $J$.

\medskip

We say that a subset $X$ of $V(G)$ \emph{saturates} the strip system
if for every $u \in V(J)$ there is at most one neighbor $v$ of $u$
such that $N_{uv} \not\subseteq X$.  A vertex $v$ in $V(G)\setminus
V(S,N)$ is \emph{major} with respect to the strip system if the set of
its neighbors saturates the strip system.  A vertex $v\in
V(G)\setminus V(S,N)$ is {\em major with respect to some choice of
rungs} if the $J$-strip system defined by this choice of rungs is
saturated by the set of neighbors of $v$.

A subset $X$ of $V(S,N)$ is \emph{local} with respect to
the strip system if either $X \subseteq N_v$ for some $v \in V(J)$ or
$X \subseteq S_{uv}$ for some edge $uv \in E(J)$.

\begin{lemma}\label{SPGT85var}
Let $G$ be a Berge graph, let $J$ be a $3$-connected graph, and let
$(S,N)$ be a $J$-strip system in $G$.  Assume moreover that if $J=K_4$
then $(S,N)$ is non-degenerate and that no vertex is major for some
choice of rungs and non-major for another choice of rungs.  Let
$F\subseteq V(G)\setminus V(S,N)$ be connected and such that no member
of $F$ is major with respect to $(S,N)$.  If the set of attachments of
$F$ in $V(S,N)$ is not local, then one can find in polynomial time one
of the following:
\begin{itemize}
\item 
A  path $P$, with $\emptyset\neq V(P)\subseteq V(F)$,
such that $V(S,N)\cup V(P)$ induces a $J$-strip system. 
\item 
A  path $P$, with $\emptyset\neq V(P)\subseteq V(F)$, and for
each edge $uv\in E(J)$ a $uv$-rung $R_{uv}$, such that $V(P)\cup
\bigcup_{uv\in E(J)} R_{uv}$ is the line-graph of a bipartite
subdivision of a $J$-enlargement.
\end{itemize}
\end{lemma}
\begin{proof}
The proof of this lemma is essentially the same as the proof of
Theorem~8.5 in \cite{CRST}.  In 8.5 there is an assumption that there
is no overshadowed appearance of $J$; but all that is used is that no
vertex is major for some choice of rungs of $(S,N)$ and non-major for
another.
\end{proof}


We say that a $K_4$-strip system $(S,N)$ in a graph $G$ is {\em
special} if it satisfies the following properties, where for all
$i,j\in [4]$, $O_{ij}$ denotes the set of vertices in $V(G)\setminus
V(S,N)$ that are complete to $(N_i\cup N_j)\setminus S_{ij}$ and
anticomplete to $V(S,N)\setminus (N_i\cup N_j\cup S_{ij})$:
\begin{itemize}
\item[(a)] 
$N_{13}=N_{31}=S_{13}$ and $N_{24}=N_{42}=S_{24}$.
\item[(b)] 
Every rung in $S_{12}$ and $S_{34}$ has even length at least $2$, and
every rung in $S_{14}$ and $S_{23}$ has odd length at least $3$. 
\item[(c)] 
$O_{12}$ and $O_{34}$ are both non-empty and complete to each other.
\item[(d)] 
If some vertex of $V(G)\setminus (V(S,N) \cup O_{12}\cup O_{34})$ is
major with respect to some choice of rungs in $(S,N)$, then it is
major with respect to $(S,N)$.  In particular, there is no
overshadowed appearance of $(S,N)$ in $G\setminus (M\cup O_{12}\cup
O_{34})$, where $M$ is the set of vertices that are major with respect
to $(S,N)$.  
\item[(e)] 
There is no enlargement of $(S,N)$ in $G\setminus (O_{12}\cup
O_{34})$, and $(S,N)$ is maximal in $G\setminus (O_{12}\cup O_{34})$.
\end{itemize}

\begin{lemma}\label{specialcase}
Let $G$ be Berge and square-free, and let $J$ be a 3-connected graph.
Let $(S,N)$ be a $J$-strip system in $G$.  Let $m\in V(G)\setminus
V(S,N)$.  If $m$ is major for some choice of rungs in $(S,N)$, then
one of the following holds:
\begin{enumerate}
\item 
There is a $J$-enlargement with a non-degenerate appearance in $G$ 
(and such an appearance can be found in polynomial time).
\item 
There is a $J$-strip system $(S',N')$ such that $V(S,N)\subset
V(S',N')$ with strict inclusion (and $(S',N')$ can be found in
polynomial time).
\item 
$m$ is major with respect to  $(S,N)$.
\item 
$G$ has a special $K_4$-strip system.
\end{enumerate}
\end{lemma}

\begin{proof}
Let $m$ be major for some choice of rungs in $(S,N)$.  Suppose that
there is no $J$-enlargement with a non-degenerate appearance in $G$,
and $(S,N)$ is maximal in $G$, and that $m$ is not major with respect
to $(S,N)$.  Let $X$ be the set of neighbors of $m$.  Let $M$ be the
set of vertices of $V(G)\setminus V(S,N)$ that are major with respect
to $(S,N)$.  Let $M^*$ be the set of vertices of $V(G)\setminus
V(S,N)$ that are major with respect to some choice of rungs.  So $m\in
M^*\setminus M$.
	
As noted earlier, every degenerate appearance of any $3$-connected
graph contains a square, so $G$ contains no degenerate appearances of
any $3$-connected graph.  Hence, by 8.4 in \cite{CRST}, we must have
$J=K_4$.  Let $V(J) = \{1,2,3,4\}$.  Since $m$ is major with respect
to some choice of rungs and not major with respect to the strip
system, we may choose rungs $R_{ij}, R'_{ij}$ ($i \neq j \in
\{1,2,3,4\}$) forming line graphs $L(H)$ and $L(H')$ respectively, so
that $X$ saturates $L(H)$ but not $L(H')$.  Moreover, we may assume
that $R_{ij} \neq R'_{ij}$ if and only if $\{i,j\} = \{1,2\}$.
	
Let the ends of each $R_{ij}$ be $r_{ij}$ and $r_{ji}$, where
$\{r_{ij}\mid j \in \{1,2,3,4\} \setminus \{i\}\}$ is a triangle $T_i$
for each $i$.  Similarly, let the ends of each $R'_{ij}$ be $r'_{ij}$
and $r'_{ji}$, where $\{r'_{ij}\mid j \in \{1,2,3,4\} \setminus
\{i\}\}$ is a triangle $T'_i$ for each $i$.
	
Since $X$ saturates $L(H)$, it has at least two members in each of
$T_1, \ldots, T_4$, and since $X$ does not saturate $L(H')$, there is
some $T'_i$ that contains at most one member of $X$.  Since $T_3 =
T'_3$ and $T_4 = T'_4$ we may assume that $|X \cap T_1| = 2$ and $|X
\cap T'_1| = 1$, so $r_{12} \in X$, $r'_{12} \not\in X$, and exactly
one of $r_{13}, r_{14}$ is in $X$, say $r_{13}\in X$ and $r_{14}
\not\in X$.  Also, at least two vertices of $T_3$ are in $X$, and
similarly for $T_4$, so there are at least two branch-vertices of $H'$
incident in $H'$ with more than one edge in $X$.  By 5.7 in
\cite{CRST} applied to $H'$, we deduce that 5.7.5 in \cite{CRST}
holds, so (since odd branches of $H'$ correspond to even rungs in
$L(H')$ and vice-versa) there is an edge $ij$ of $J$ such that
\begin{equation}\label{lg:xp}
\mbox{$R'_{ij}$ is even and $[X \cap V(L(H'))] \setminus V(R'_{ij}) =
(T'_i \cup T'_j) \setminus V(R'_{ij})$.}
\end{equation}
In particular, $T'_i$ and $T'_j$ both contain at least two vertices in
$X$, so $i, j \geq 2$.  Since $r_{13} \in X$, it follows that one of
$i,j$ is equal to $3$, say $j=3$, and so $r_{13}=r_{31}$, in other
words $R_{13}$ has length $0$.  Hence $i\in\{2,4\}$.  We claim that:
\begin{equation}\label{lg:i4}
i=4.
\end{equation}
For suppose that $i=2$.  By (\ref{lg:xp}) $R_{23}$ is even and $[X
\cap V(L(H'))]\setminus V(R_{23}) = \{r'_{21}, r_{24}, r_{31},
r_{34}\}$.  Since at least two vertices of $T_4$ are in $X$ it follows
that $r_{42}=r_{24}$ and $r_{43}=r_{34}$ (and $r_{41}\notin X$).
Hence $R_{24}$ and $R_{34}$ both have length $0$, and since $R_{23}$
is even this is a contradiction to the last axiom in the definition of
a strip system.  Thus (\ref{lg:i4}) holds.

\medskip
	
Therefore we have $i=4$ and $j=3$.  So (\ref{lg:xp}) translates to:
\begin{equation}\label{lg:ss}
\mbox{$R_{34}$ is even and $[X \cap V(L(H'))]\setminus V(R_{34}) =
\{r_{31}, r_{32}, r_{41}, r_{42}\}$.}
\end{equation}
This implies that $V(R'_{12})\cap X=\emptyset$; moreover, if
$r_{23}\in X$ then $r_{23}=r_{32}$, and similarly if $r_{24}\in X$
then $r_{24}=r_{42}$.
\begin{equation}\label{lg:R14}
\longbox{One of $R_{23}, R_{24}$ has length $0$, the other has odd
length, $R_{14}$ has odd length, and $r_{21}\in X$.}
\end{equation}
Since the path $r_{32}$-$R_{23}$-$r_{23}$-$r_{24}$-$R_{24}$-$r_{42}$
can be completed to a hole via
$r_{42}$-$r_{43}$-$R_{34}$-$r_{34}$-$r_{32}$, the first path is even,
and so exactly one of $R_{23}, R_{24}$ is odd.  Since the same path
can be completed to a hole via
$r_{42}$-$r_{41}$-$R_{14}$-$r_{14}$-$r_{13}$-$r_{32}$, it follows that
$R_{14}$ is odd.  Since one of $R_{23}, R_{24}$ is odd, they do not
both have length $0$, and hence at most one of $r_{23}, r_{24}$ is in
$X$.  On the other hand, since $X$ saturates $L(H)$, the triangle
$T_2$ has at least two vertices from $X$; it follows that $r_{21} \in
X$ and that exactly one of $r_{23},r_{24}$ is in $X$ (in other words
exactly one of $R_{23}, R_{24}$ has length $0$).  Thus (\ref{lg:R14}) 
holds.

\begin{equation}\label{lg:r21}
\mbox{$R_{12}$ has length $0$.}
\end{equation}
For suppose that $r_{21} \neq r_{12}$.  If $r_{21}$ has a neighbor in
$R'_{12}$, then $m$ can be linked onto the triangle $T'_1$ via
$R'_{12}$, $R_{14}$ and $m$-$r_{13}$, a contradiction.  Hence $r_{21}$
has no neighbor in $R'_{12}$.  Then from the hole
$m$-$r_{21}$-$r_{24}$-$r'_{21}$-$R'_{1,2}$-$r'_{12}$-$r_{13}$-$m$, we
deduce that the rungs $R_{12}$ and $R'_{12}$ are odd.  But then either
$m$-$r_{21}$-$r_{23}$-$r'_{21}$-$R'_{12}$-$r'_{12}$-$r_{14}$-$R_{14}$-$r_{41}$-$m$
or
$m$-$r_{21}$-$r_{24}$-$r'_{21}$-$R'_{12}$-$r'_{12}$-$r_{14}$-$R_{14}$-$r_{41}$-$m$
is an odd hole, contradiction.  Thus (\ref{lg:r21}) holds.  It follows
that every $12$-rung (in particular $R'_{12}$) has even length.

\begin{equation}\label{lg:r24}
\mbox{$R_{24}$ has length $0$ and $R_{23}$ has odd length.}
\end{equation}
For suppose the contrary.  As shown above, this means that $R_{23}$
has length $0$ and $R_{24}$ has odd length.  Then $R_{24}$, $R_{12}$
and $R_{14}$ contradict the last axiom in the definition of a strip
system (the parity condition).  Thus (\ref{lg:r24}) holds.  So
$r_{24}=r_{42}$ and $r_{23}\neq r_{32}$ (and hence $r_{23}\notin X$).

\begin{equation}\label{lg:R34}
\mbox{Every $34$-rung has non-zero even length.}
\end{equation}
By (\ref{lg:ss}) $R_{34}$ has even length, so every $34$-rung has even
length.  If some $34$-rung has length zero, then its unique vertex $x$
is such that $\{x,r_{42},r_{21},r_{13}\}$ induces a square, a
contradiction.  Thus (\ref{lg:R34}) holds.

\medskip
			
For $i \neq j$, let $O_{ij}$ be the set of vertices that are not major
with respect to $L(H')$ and are complete to $(T'_i \cup T'_j)\setminus
R'_{ij}$.  In particular, $r_{12}$ ($=r_{21}$) is in $O_{12}$ and $m$
is in $O_{34}$, so $O_{12}$ and $O_{34}$ are non-empty.  Every vertex
in $M^*\setminus M$ is complete to $\{r_{13},r_{32},r_{42}, r_{41}\}$
and has no other neighbor outside of $R_{34}$ in $L(H)$.  Moreover,
since $G$ is square-free, every such vertex is adjacent to every
$12$-rung of length $0$.

\begin{equation}\label{eq:alter}
\longbox{For $\{i,j\} \notin \{\{ 1,2\},\{ 3,4\}\}$ and for every rung
$R$ in $S_{ij}$ let $L(H_1)$ (resp.~$L(H_1')$) be the graph obtained
from $L(H)$ (resp.~$L(H')$) by replacing $R_{ij}$ with $R$.  Then $m$
is major with respect to $L(H_1)$ and non-major with respect to
$L(H_1')$.}
\end{equation}
Clearly $m$ is non-major with respect to $L(H_1')$.  Suppose it is
also non-major with respect to $L(H_1)$.  Then by symmetric argument
applied to $L(H)$ and $L(H_1)$, it follows that $R$ is of even length.
So we may assume that $\{ i,j\} =\{ 1,3\}$.  But then $R_{24}$ must be
of non-zero length, a contradiction.  Thus (\ref{eq:alter}) holds.

\medskip

By (\ref{eq:alter}) all rungs in $S_{13}$ and $S_{24}$ have length
$0$, and all rungs in $S_{23}$ and $S_{14}$ are odd.  Also,
$M^*\setminus M$ is complete to $N_{13} \cup N_{32} \cup N_{42} \cup
N_{41}$ and to every zero-length rung in $S_{12}$ and has no other
neighbor in $V(S,N) \setminus S_{34}$.  Thus $M^*\setminus M\subseteq
O_{34}$; and conversely, since $O_{34}$ is complete to $R_{12}$ (for
otherwise $O_{34}\cup \{r_{12},r_{24},r_{13}\}$ would contain a
square), we deduce that $O_{34}=M^*\setminus M$.  We observe that if
$R$ is any $14$-rung or $23$-rung, then $R$ has length at least $3$,
for otherwise $R$ has length $1$ and $V(R)\cup\{r_{21},m\}$ induces a
square.

\medskip

Let $(S',N')$ be the strip system obtained from $(S,N)$ by replacing
$S_{12}$ with $S_{12} \setminus O_{12}$.  It follows from the
definition of $(S',N')$ and the facts above that items (a)--(c) of the
definition of a special $K_4$-strip system hold.  Since only $S_{12}$
and $S_{34}$ have even non-zero rungs, we deduce that item~(d) in that
definition also holds.

Finally suppose that $(S',N')$ is not maximal in $G \setminus (O_{12}
\cup O_{34})$.  Since there is no $J$-enlargement of $(S,N)$ and
$(S,N)$ is maximal, there exists an appearance $(S'',N'')$ of $J$ that
contains $(S',N')$, and we may assume that $(S'',N'')$ is obtained
from $(S',N')$ by adding one rung $R$.  If $R \in S''_{12}$, then
$(S'',N'')$ is an enlargement of $(S,N)$, a contradiction.  So $R \not
\in S''_{12}$, and we do not get a $J$-enlargement by adding $O_{12}
\cap S_{12}$ to $S''_{12}$.  Therefore, there is $r \in O_{12} \cap
S_{12}$ such that we do not get a $J$-enlargement or a larger strip by
adding $r$ to $S''_{12}$.  By 5.8 of \cite{CRST}, $r$ is major with
respect to an appearance of $J$ using the new rung, and non-major
otherwise.  So $R \in S''_{34}$, $|V(R)|=1$ and $V(R) \subseteq
O_{34}$, a contradiction.  Thus, $(S,N)$ is a special $K_4$-strip
system in $G$, and outcome~4 of the theorem holds.
\end{proof}

We now focus on the case of a special $K_4$-strip system.
\begin{lemma}\label{specialcase2}
Let $G$ be a square-free Berge graph and $(S,N)$ be a special
$K_4$-strip system in $G$, with the same notation as in the
definition.  Let $M$ be the set of vertices that are major with
respect to $(S,N)$.  Let $X_1 \in \{N_{12}, N_1\setminus N_{12}\}, X_2
\in \{N_{21}, N_2\setminus N_{21}\}$ and $X = X_1\cup X_2$.  Let $A =
S_{12}\setminus X$ and $B = V(S,N)\setminus (S_{12}\cup X)$.  Let $F
\subseteq V(G)\setminus (V(S,N)\cup M\cup O_{12})$ be connected.  Then
$F$ has attachments in at most one of $A$ and $B$.
\end{lemma}
\begin{proof}
Suppose for the sake of contradiction that $F$ has attachments in both
$A$ and $B$.  We may assume that $|F|$ is minimal under this
condition.  Then $F$ forms a path with ends $f_1, f_2$ such that $f_1$
has attachments in $A$, $f_2$ has attachments in $B$, and there are no
other edges between $F$ and $A\cup B$.

Let $Y$ be the set of attachments of $F$ in $V(S,N)$.  Suppose that
$Y$ is local with respect to $(S,N)$.  Then, as $F$ has attachments in
both $A \subseteq S_{12}$ and $B \subseteq V(S,N)\setminus S_{12}$, it
follows that either $Y \subseteq N_1$ or $Y \subseteq N_2$.  We may
assume without loss of generality that $Y \subseteq N_1$.  Then $N_1
\cap A$ is non-empty, so $N_{12} \not \subseteq X$, and $N_1 \cap B$
is non-empty, so $N_1\setminus N_{12} \not \subseteq X$, a
contradiction.  Hence $Y$ is not local in $(S,N)$.
		
Suppose that $F \cap O_{34} \neq \emptyset$.  Then $f_2 \in O_{34}$.
Let $(S',N')$ be the strip system obtained from $(S,N)$ by adding
$O_{34}$ to $S_{34}$.  Then $F \setminus \{f_2\}$ has non-local
attachments in $(S',N')$, and no vertex of $F \setminus \{f_2\}$ has
neighbors in $B$.  Let $L(H)$ be the line graph formed by some choice
of rungs in $(S',N')$, where $f_2$ is the rung chosen from $S_{3,4}$,
and the rung from $S_{1,2}$ contains a neighbor of $f_1$.  Apply 5.8
of \cite{CRST}.  Since no vertex of $F \setminus \{f_2\}$ has a
neighbor in $B \setminus \{f_2\}$, none of the outcomes are possible,
a contradiction.  This proves that $F \cap O_{34} = \emptyset$.  So
$F\subseteq V(G)\setminus (V(S,N)\cup M\cup O_{12}\cup O_{34})$.  By
Lemma~\ref{specialcase}, $(S,N)$ is maximal in $G \setminus (O_{12}
\cup O_{34})$, and no vertex of $V(G) \setminus (V(S,N) \cup M \cup
O_{12} \cup O_{34})$ is major or overshadowing with respect to
$(S,N)$, a contradiction to Lemma~\ref{SPGT85var}.  This proves the
theorem.
\end{proof}	

\begin{lemma}\label{specialcase3}
Let $G$ be a square-free Berge graph and $(S,N)$ be a special
$K_4$-strip system in $G$, with the same notation as in the
definition.  Let $M$ be the set of vertices that are major with
respect to $(S,N)$.  Then:
\begin{enumerate}[label = \emph{(\arabic*)}]	
\item 
$O_{12}\cup M$ and $O_{34} \cup M$ are cliques; and
\item 
there is an integer $k$ such that $O_{12}\cup M \cup (N_1 \setminus
N_{1k})$ is a clique, and similarly there is an integer $\ell$ such
that $O_{12}\cup M \cup (N_2\setminus N_{2\ell})$ is a clique.
\end{enumerate}
\end{lemma}
\begin{proof}
Suppose that (1) does not hold.  Then there are non-adjacent vertices
$x_1,x_2$ in $O_{12}\cup M$, say.  If $x \in O_{12}$, then by
Lemma~\ref{specialcase} $x$ is complete to $N_{1k}$ for all $k \neq
2$, and complete to $N_{2\ell}$ for all $\ell \neq 1$.  If $x \in M$,
then $x$ is complete to $N_{1k}$ for all but at most one $k$, and
complete to $N_{2\ell}$ for all but at most one $\ell$.  Hence there
exist $k, \ell$ so that $\{x_1,x_2\}$ is complete to $N_{1k}\cup
N_{2\ell}$, so for every $u \in N_{1k}$ and $v \in N_{2\ell}$, $\{x_1,
u, x_2, v\}$ induces a square, contradiction.  This proves~(1).

By definition, for every $x \in O_{12}\cup M$ there are indices $k$
and $\ell$ so that $x$ is complete to $(N_1\setminus N_{1k}) \cup
(N_2\setminus N_{2\ell})$.  Hence (2) follows from (1) by a direct
application of Lemma \ref{C4cliques}.
\end{proof}

\begin{lemma}\label{specialgood}
Let $G$ be a square-free Berge graph.  If $G$ has a special
$K_4$-strip system, then it has a good partition.
\end{lemma}
\begin{proof}
Let $(S,N)$ be a special $K_4$-strip system of $G$, with the same
notation as above.  Let $M$ be the set of vertices that are major with
respect to $(S,N)$.  There are vertices $t_{12}\in S_{12}\setminus
(N_{12}\cup N_{21})$, $t_{34}\in S_{34}\setminus (N_{34}\cup N_{43})$
and $t_{13}\in S_{13}$, and hence $\{t_{12}, t_{34}, t_{13}\}$ is a
triad $\tau$.

Suppose that both $(N_1 \setminus N_{12}) \cup M \cup O_{12}$ and
$(N_2 \setminus N_{21}) \cup M \cup O_{12}$ are cliques.  Let $K_1=N_1
\setminus N_{12}$, $K_2=O_{12} \cup M$, and $K_3=N_2 \setminus
N_{21}$.  By Lemma~\ref{specialcase2}, $K_1\cup K_2\cup K_3$ is a
cutset.  Let $L$ be the union of those components of $G\setminus
(K_1\cup K_2\cup K_3)$ that contain vertices of $S_{12}$, and let
$R=V(G) \setminus (L \cup K_1 \cup K_2 \cup K_3)$.  Then $K_1$ is
anticomplete to $K_3$, and every  path from $K_3$ to $K_1$
with interior in $L$ contains a vertex of $N_{12}$, which is complete
to $K_1$, and $\tau$ is a triad that contains a vertex of $L$ and a
vertex of $R$.  So $(K_1, K_2, K_3, L, R)$ is a good partition of
$V(G)$.

Now assume, up to symmetry, that $(N_1 \setminus N_{12}) \cup M \cup
O_{12}$ is not a clique.  By Lemma~\ref{specialcase3}, $N_{12} \cup M
\cup O_{12}$ is a clique.  Also, at least one of $N_{21} \cup M \cup
O_{12}$ and $(N_2 \setminus N_{21}) \cup M \cup O_{12}$ is a clique.
If the former is a clique, let $X=N_{21}$, and otherwise let $X=N_2
\setminus N_{21}$.  Set $K_1=N_{12}$, $K_2=M \cup O_{12}$, and
$K_3=X$.  By Lemma~\ref{specialcase2}, $K_1\cup K_2\cup K_3$ is a
cutset.  Let $L$ be the component of $G\setminus (K_1\cup K_2\cup
K_3)$ that contains $N_1 \setminus N_{12}$ (note that $N_1\setminus
N_{12}$ is connected because $N_{13}$ is complete to $N_{14}$), and
let $R=V(G) \setminus (L \cup K_1 \cup K_2 \cup K_3)$.  Then $K_1$ is
anticomplete to $K_3$, and every path from $K_3$ to $K_1$ with
interior in $L$ contains a vertex of $N_1 \setminus N_{12}$, which is
complete to $K_1$, and $\tau$ is a triad that contains a vertex of $L$
and a vertex of $R$.  So $(K_1, K_2, K_3, L, R)$ is a good partition
of $V(G)$.
\end{proof}


Now we can give the proof of Theorem \ref{thm:lgb}.

\begin{proof}
Since $G$ contains the line-graph of a bipartite subdivision of $K_4$,
there is a $3$-connected graph $J$ such that $G$ contains an
appearance of $J$, and we choose $J$ maximal with this property.
Hence $G$ contains the line-graph $L(H)$ of a bipartite subdivision
$H$ of $J$.  Then there exists a $J$-strip system $(S,N)$ such that
$V(S,N)\subseteq V(G)$, and we choose $V(S,N)$ maximal.  Let $M$ be
the set of vertices in $V(G)\setminus V(S,N)$ that are major with
respect to the strip system $(S,N)$.  We observe that:
\begin{equation}\label{eq:lgb1} 
\mbox{$M$ is a clique.}
\end{equation}
Suppose that $m,m'$ are non-adjacent vertices in $M$.  Let $B$ be a
branch of $H$, and let $u,v$ be its ends.  Since there is no triangle
in $H$, there exist a neighbor $u'$ of $u$ and a neighbor $v'$ of $v$
in $H$ such that $N_{uu'}$ and $N_{vv'}$ are complete to $M$ and
anticomplete to each other.  Then $\{m,m',u',v'\}$ induces a square.
This proves (\ref{eq:lgb1}).
\begin{equation}\label{eq:lgb2} 
\longbox{For every branch vertex $u$ in $H$, there is a branch vertex
$v$ in $H$ such that $M\cup (N_u\setminus N_{uv})$ is a clique.}
\end{equation}
It follows from (\ref{eq:lgb1}) that $M$ is a clique, and by the
definition of major vertices, for every $m \in M$ and every branch
vertex $u$ there is a branch vertex $v$ such that $m$ is complete to
$N_u\setminus N_{uv}$.  Hence (\ref{eq:lgb2}) follows by a direct
application of Lemma~\ref{C4cliques}.

\medskip
 		
If some vertex of $V(G)\setminus V(S,N)$ is major with respect to some
choice of rungs but not with respect to the strip system, then by
Lemma~\ref{specialcase} $G$ has a special $K_4$-strip system, and by
Lemma~\ref{specialgood} $G$ has a good partition, so the theorem
holds.  Therefore we may assume that every vertex of $V(G)\setminus
V(S,N)$ that is major with respect to some choice of rungs is major
with respect to the strip system.  By Lemma~\ref{SPGT85var} (or
Theorem~8.5 from \cite{CRST}), every component of
$V(G)\setminus(V(S,N)\cup M)$ attaches locally to $V(S,N)$. 
\begin{equation}\label{triad}
\longbox{For every strip $S_{uv}$ there exists a triad $\{t,t',t''\}$
in $G$ such that $t\in S_{uv}$ and $t',t''\in V(S,N) \setminus
(S_{uv}\cup N_u\cup N_v)$.}
\end{equation}
For every strip $S_{xy}$ let $R_{xy}$ be an $xy$-rung, with
endvertices $r_{xy}\in N_{xy}$ and $r_{yx}\in N_{yx}$.  Suppose that
$r_{uv}\neq r_{vu}$.  Since $J$ is $3$-connected, there is a cycle $C$
in $J$ that contain $u$ and not $v$.  In $G$ let $C' = \bigcup_{xy\in
E(C)} R_{xy}$.  Then $C'$ is a hole of length at least~$6$, and it is
even since $G$ is Berge, so it has two non-adjacent vertices $t',t''$
that are not in $N_u$.  Then $\{r_{vu}, t', t''\}$ is the desired
triad.  Now suppose that $r_{uv}= r_{vu}$.  There is a cycle $C$ in
$J$ that contains $u$ and $v$.  In $G$ let $C' = \bigcup_{xy\in E(C)}
R_{xy}$.  Then $C'$ is an even hole, of length at least~$6$, so it has
three non-adjacent vertices including $r_{uv}$.  Then these vertices
form the desired triad.  So (\ref{triad}) holds.

\medskip

For every strip $S_{uv}$, let $S_{uv}^*$ denote the union of $S_{uv}$
with the components of $G \setminus V(S,N)$ that attach in $S_{uv}$
only, and let $T_{uv}=N_u \cap N_v$ ($=N_{uv}\cap N_{vu}$).  Note that
$T_{uv}$ is complete to $N_u\setminus N_{uv}$ and to $N_v\setminus
N_{vu}$.  Moreover we observe that:
\begin{equation}\label{MTuv}
\mbox{$M\cup T_{uv}$ is a clique.}
\end{equation}
Suppose that $M\cup T_{uv}$ has two non-adjacent vertices $a,b$.  By
(\ref{eq:lgb2}), and since every branch vertex in $H$ has degree at
least~$3$, $M$ is complete to at least one vertex $n_u\in N_u
\setminus N_{uv}$, and similarly to at least one vertex $n_v\in N_v
\setminus N_{vu}$.  By (\ref{eq:lgb1}) at least one of $a,b$ is in
$T_{uv}$, say $a\in T_{uv}$.  Since edges in $H$ that correspond to
$a,n_u$ and $n_v$ cannot induce a triangle (as $H$ is bipartite), it
follows that $n_u$ and $n_v$ are not adjacent.  Then $\{a,b,n_u,n_v\}$
induces a square, a contradiction.  So (\ref{MTuv}) holds.

\medskip

Let us say that a strip $S_{uv}$ is {\em rich} if $S_{uv} \setminus
T_{uv} \neq \emptyset$.

\begin{equation}\label{ifrich}
\mbox{If $(S,N)$ has a rich strip, the theorem holds.}
\end{equation}
Let $S_{uv}$ be a rich strip in $(S,N)$.  First suppose that both $M
\cup (N_u \setminus N_{uv})$ and $M \cup (N_v \setminus N_{vu})$ are
cliques.  Hence, by (\ref{MTuv}) and the definition of $T_{uv}$, both
$M \cup (N_u \setminus N_{uv})\cup T_{uv}$ and $M \cup (N_v \setminus
N_{vu})\cup T_{uv}$ are cliques.  Let $K_1=N_u \setminus N_{uv}$,
$K_2=M \cup T_{uv}$, $K_3=N_v \setminus N_{vu}$, let $L$ consist of
$S^*_{uv}\setminus T_{uv}$ together with those components of
$G\setminus V(S,N)$ that attach only to $N_u$ and those that attach
only to $N_v$, and let $R=V(G) \setminus (K_1 \cup K_2 \cup K_3 \cup
L)$.  Then every  path from $K_3$ to $K_1$ with interior in
$L$ contains a vertex of $N_{uv}$, which is complete to $K_1$, and no
vertex of $L$ has both a neighbor in $K_1$ and a neighbor in $K_3$;
moreover, by (\ref{triad}) there is a triad $\{t, t', t''\}$ with
$t\in S_{uv}$ and $t',t''\in V(S,N)\setminus (S_{uv}\cup N_u\cup
N_v)$, so this is a triad with a vertex (namely $t$) in $L$ and a
vertex in $R$; so $(K_1, K_2, K_3, L, R)$ is a good partition of
$V(G)$.  \\
Therefore we may assume that $M \cup (N_u \setminus N_{uv})$ is not a
clique, and so $M \cup N_{uv}$ is a clique.  If $M \cup (N_v \setminus
N_{vu})$ is a clique, let $K_1=N_{uv}\setminus T_{uv}$, $K_2=M \cup
T_{uv}$, $K_3=N_v \setminus N_{vu}$, let $R$ consist of $S_{uv}^*
\setminus N_u$ together with those components of $G\setminus V(S,N)$
that attach only to $N_v$, and let $L=V(G) \setminus (R \cup K_1 \cup
K_2 \cup K_3)$.  Then $K_1$ is anticomplete to $K_3$, and every
 path from $K_3$ to $K_1$ with interior in $L$ contains a
vertex of $N_u \setminus N_{uv}$, which is complete to $K_1$;
moreover, by (\ref{triad}) there is a triad $\{t, t', t''\}$ with
$t\in S_{uv}$ and $t',t''\in V(S,N)\setminus (S_{uv}\cup N_u\cup
N_v)$, so this is a triad with a vertex in $L$ and a vertex (namely
$t$) in $R$; So $(K_1, K_2, K_3, L, R)$ is a good partition of $V(G)$.
\\
Therefore we may assume that for every rich strip $S_{xy}$, both
$M\cup N_{xy}$ and $M\cup N_{yx}$ are cliques, and neither of $M\cup
(N_{x} \setminus N_{xy})$ and $M\cup (N_{y} \setminus N_{yx})$ is a
clique.  Hence, regarding $S_{uv}$, there is an edge $uw$ in $J$ such
that $M\cup N_{uw}$ is not a clique.  Then $S_{uw}$ is not rich, and
hence $S_{uw}=T_{uw}=N_{uw}$.  By (\ref{MTuv}) $M\cup T_{uw}=M\cup
N_{uw}$ is a clique, a contradiction.  So (\ref{ifrich}) holds.

\medskip

By (\ref{ifrich}) we may assume that there is no rich strip in
$(S,N)$.  It follows that for every $uv\in E(J)$ we have
$S_{uv}=T_{uv}$, which is a clique by (\ref{MTuv}).  Consequently
$N_u$ is a clique for every $u$, and by (\ref{MTuv}), $M\cup N_u$ is a
clique for every $u$.  Let $S_{uv}$ be a strip.  By (\ref{triad})
there is a triad $\{t, t', t''\}$ with $t\in S_{uv}$ and $t',t''\in
V(S,N)\setminus (S_{uv}\cup N_u\cup N_v)$.  Let $K_1=N_u \setminus
S_{uv}$, $K_2=M$, $K_3=N_v \setminus S_{uv}$, let $L$ consist of
$S_{uv}^*$ together with the components of $G\setminus V(S,N)$ that
attach only to $N_u$ and only to $N_v$, and let $R=V(G) \setminus (K_1
\cup K_2 \cup K_3 \cup L)$.  Then $K_1$ is anticomplete to $K_3$
(since there is no triangle in $H$), and every  path from
$K_3$ to $K_1$ with interior in $L$ contains a vertex of $S_{uv}$,
which is complete to $K_1$, and $\{t, t', t''\}$ is a triad with a
vertex in $L$ and a vertex in $R$.  So $(K_1, K_2, K_3, L, R)$ is a
good partition of $V(G)$.  This concludes the proof.
\end{proof}	

\section{Algorithmic aspects}

Assume that we are given a graph $G$ on $n$ vertices.  We want to know
if $G$ is a square-free Berge graph and, if it is, we want to produce
an $\omega(G)$-coloring of $G$.  We can do that as follows, based on
the method described in the preceding sections.  We can first test
whether $G$ is square-free in time $O(n^4)$.  Therefore let us assume
that $G$ is square-free.

Let $\cal A$ be the class of graphs that contain no odd hole, no
antihole of length at least $6$, and no prism (sometimes called
``Artemis'' graphs).  There is an algorithm, ``Algorithm 3'' in
\cite{MT}, of time complexity $O(n^9)$, which decides whether the
graph $G$ is in class~$\cal A$ or not, and, if it is not, returns an
induced subgraph of $G$ that is either an odd hole, an antihole of
length at least $6$, or a prism.  If the first outcome happens, then
$G$ is not Berge and we stop.  The second outcome cannot happen since
$G$ is square-free.  Therefore we may assume that $G$ is Berge and
that the algorithm has returned a prism $K$.  We want to extend $K$
either to a maximal hyperprism or to the line-graph of a bipartite
subdivision of $K_4$.  We can do that as follows.  Let $K$ have rungs
$R_1$, $R_2$, $R_3$, where, for each $i=1,2,3$, $R_i$ has ends $a_i,
b_i$, such that $\{a_1, a_2, a_3\}$ and $\{b_1, b_2, b_3\}$ are
triangles.
\begin{itemize}
\item
Initially, for each $i\in\{1,2,3\}$ let $A_i=\{a_i\}$, $B_i=\{b_i\}$
and $C_i=V(R_i)\setminus\{a_i, b_i\}$.  Let $V(H)=V(K)$.
\item
Let $M$ be the set of major neighbors of $H$.
\item
If there is a component $F$ of $G\setminus (H\cup M)$ whose set of
attachments on $H$ is not local, then by Lemma~\ref{lem:hyplocal}, one
of the following occurs (and can be found in polynomial time):
\begin{itemize}
\item[(i)] 
There is a path $P$ in $F$ such that $V(H)\cup V(P)$ induces
a larger hyperprism $H'$; or
\item[(ii)] 
There are three rungs $R_1, R_2, R_3$ of $H$, one in each strip of
$H$, and a  path $P$ in $F$, such that $V(R_1)\cup V(R_2)\cup
V(R_3)\cup V(P)$ induces the line-graph of a bipartite subdivision of
$K_4$.
\end{itemize}
\end{itemize}
Assume that outcome (ii) never happens.  Whenever outcome (i) happens,
we start again from the larger hyperprism that has been found.  Note
that outcome (i) can happen only $n$ times, because at each time we
start again with a strictly larger hyperprism.  So the procedure
finishes with a maximal hyperprism.  Then we can find a good partition
of $G$ as explained in Theorem~\ref{thm:epr} or~\ref{thm:opr},
decompose $G$ along that partition, and color $G$ using induction as
explained in Lemma~\ref{lem:cutset}.

\medskip 

Remark:  Since a hyperprism may have exponentially many rungs, we
need to show how  we can determine in polynomial time the set $M$ of
major neighbors of a hyperprism $H$ in a graph $G$ without listing
all the rungs of $H$. It is easy to see that a vertex $x$ in
$V(G)\setminus V(H)$ is a major neighbor of $H$ if and only if one of
the following two situations occurs: 
\begin{itemize}
\item 
For at least two distinct values of $i\in\{1,2,3\}$, there exists an
$i$-rung $R_i$ such that $x$ is adjacent to both ends of $R_i$, or
\item
For a permutation $\{i,j,k\}$ of $\{1,2,3\}$, there exists an
$i$-rung $R_i$ such that $x$ adjacent to both ends of $R_i$ and $x$
has a neighbor in $A_j$ and a neighbor in $B_k$.
\end{itemize}
So it suffices to test, for each $i\in\{1,2,3\}$, whether there exists
an $i$-rung such that $x$ is adjacent to both its ends.  This can be
done as follows.  For every pair $u_i\in A_i$ and $v_i\in B_i$, test
whether there is a  path between $u_i$ and $v_i$ in the
subgraph induced by $C_i\cup\{u_i,v_i\}$.  If there is any such path
$R_i$, then record it for the pair $\{u_i,v_i\}$, and for every vertex
$x$ in $V(G)\setminus V(H)$ record whether $x$ is adjacent to both
$u_i$ and $v_i$ or not.  This takes time $O(n^4)$ ($O(n^2)$ for each
pair $\{u_i,v_i\}$).  So the whole procedure of growing the hyperprism
and determining the set $M$ of its major neighbors takes time
$O(n^4)$.

\bigskip

Now assume that outcome (ii) happens, and so $G$ contains the
line-graph of a bipartite subdivision of $K_4$.  So $G$ contains the
line-graph of a bipartite subdivision of a $3$-connected graph $J$,
and we want to grow $J$ and the corresponding $J$-strip system $(S,N)$
to maximality.  We can do that as follows.
\begin{itemize}
\item
Initially, let $(S,N)$ be the strip system equal to the line-graph of 
a bipartite subdivision of $K_4$ found in outcome (ii). 
\item
Let $M$ be the set of vertices in $V(G)\setminus V(S,N)$ that are
major on some choice of rungs of $(S,N)$.  (Determining $M$ can be
done with the same arguments as in the remark above concerning the set
of major neighbors of a hyperprism, and we omit the details.)
\item
If there is a component $F$ of $G\setminus (V(S,N)\cup M)$ whose set
of attachments on $H$ is not local, then by Lemma~\ref{SPGT85var}, one
of the following occurs (and can be found in polynomial time):
\begin{itemize}
\item 
A  path $P$, with $\emptyset\neq V(P)\subseteq V(F)$, such
that $V(S,N)\cup V(P)$ induces a $J$-strip system, or
\item 
A  path $P$, with $\emptyset\neq V(P)\subseteq V(F)$, and for
each edge $uv\in E(J)$ a $uv$-rung $R_{uv}$, such that $V(P)\cup
\bigcup_{uv\in E(J)} R_{uv}$ is the line-graph of a bipartite
subdivision of a $J$-enlargement.
\end{itemize}
\item
If some vertex in $M$ is not major on some choice of rungs of $(S,N)$,
then, by Lemma~\ref{specialcase}, we can either find a larger strip
system or the special case described in item (iv) of that
lemma.

In either case, whenever we find a larger strip system we start again
with it.  This will happen at most $n$ times.  So the procedure
finishes with a maximal strip system.  Similarly to the case of the
hyperprism, the whole procedure of growing the strip system and
determining the set $M$ of its major neighbors takes time $O(n^4)$.
Then we can find a good partition of $G$ as explained in
Theorem~\ref{thm:lgb}, decompose $G$ along that partition, and color
$G$ using induction as explained in Lemma~\ref{lem:cutset}.
\end{itemize}

\bigskip

\noindent{\it Complexity analysis.} Whenever $G$ contains a prism, we
have shown that $G$ has a partition into sets $K_1$, $K_2$, $K_3$,
$L$, $R$ such that $K_1\cup K_2$ and $K_2\cup K_3$ are cliques, with
$L$ and $R$ non-empty, and $L$ is anticomplete to $R$.  Then $G$ is
decomposed into the two proper induced subgraphs $G\setminus L$ and
$G\setminus R$.  These subgraphs themselves may be decomposed, etc.
This can be represented by a decomposition tree $T$, where $G$ is the
root, and the children of every non-leaf node $G'$ are the two induced
subgraphs into which $G'$ is decomposed.  Every leaf is a subgraph
that contains no prism.

Let us consider the triads of $G$.  By item~(v) of a good partition,
there exists a triad $\tau_G$ that has at least one vertex from each
of $L,R$; we label $G$ with $\tau_G$.  Since the cutset $K_1\cup
K_2\cup K_3$ is the union of two cliques it contains no triad, and so
no triad of $G$ is in both $G\setminus L$ and $G\setminus R$; moreover
$\tau_G$ itself is in none of these two subgraphs.  Consequently every
triad of $G$ can be used as the label of at most one non-leaf node of
$T$.  So $T$ has at most $n^3$ non-leaf nodes.  Since every node has
at most two children, the number of leaves is at most $2n^3$, and the
total number of nodes of $T$ is at most $3n^3$.

Testing if $G$ is Berge takes time $O(n^9)$; this is done only once,
at the first step of the algorithm, as a subroutine of testing whether
$G$ is in class~$\cal A$.  At any decomposition node of $T$ different
from the root we already know that we have a Berge graph (an induced
subgraph of $G$), so we need only test whether the graph contains a
prism; this can be done in time $O(n^5)$ with ``Algorithm~2'' from
\cite{MT}.  The complexity of coloring a leaf (which contains
no prism) is $O(n^6)$ in \cite{MT} and $O(n^4)$ in \cite{LMRT}. The
coloring algorithm described in Lemma~\ref{lem:cutset} involves only a
few bichromatic exchanges, so its complexity is small.  The complexity
of growing a hyperprism (once a prism is known) or a strip structure
is also negligible in comparison with the rest.  So the total
complexity of the algorithm is $O(n^9)$ $+$ $O(n^3)\times O(n^4)=$
$O(n^9)$ (proving Theorem~\ref{thm:main}).


\subsection*{Acknowledgment}
We are grateful to Aur\'elie Lagoutte for useful discussions. We also
thank Cemil Dibek and Sophie Spirkl for finding a mistake in
Theorem~\ref{thm:ohyp} and for their help in correcting it.



\begin{thebibliography}{9}
    
\bibitem{Ale91}
V.E.~Alekseev.  On the number of maximal independent sets in graphs
from hereditary classes.  {\it Combinatorial-Algebraic Methods in
Discrete Optimization, University of Nizhny Novgorod,} 1991, 5--8 (in
Russian).

\bibitem{ber60}
C.~Berge.  Les probl\`emes de coloration en th\'eorie des graphes.
\emph{Publ.  Inst.  Stat.  Univ.  Paris} 9 (1960) 123--160.

\bibitem{ber61}
C.~Berge.  F\"arbung von Graphen, deren s\"amtliche bzw.~deren
ungerade Kreise starr sind (Zusammenfassung).  \emph{Wiss.  Z. Martin
Luther Univ.  Math.-Natur.  Reihe} (Halle-Wittenberg) 10 (1961)
114--115.

\bibitem{ber85}
C.~Berge.  {\em Graphs}.  North-Holland, Amsterdam/New York, 1985.

\bibitem{chibaNi85}
N.~Chiba, T.~Nishizeki.  Arboricity and subgraph listing algorithms.
{\it SIAM Journal on Computing} 14 (1985) 210--223.

\bibitem{CCLSV}
M.~Chudnovsky, G.~Cornu\'ejols, X.~Liu, P.~Seymour, K.~Vu\v{s}kovi\'c.
Recognizing Berge Graphs.  {\it Combinatorica} 25 (2005) 143--186.

\bibitem{CRST}
M.~Chudnovsky, N.~Robertson, P.~Seymour, R.~Thomas.  The strong
perfect graph theorem.  {\em Annals of Mathematics} 164 (2006)
51--229.

\bibitem{CCV}
M.~Conforti, G.~Cornu\'ejols, K.~Vu\v{s}kovi\'c.  Square-free perfect
graphs.  {\it Journal of Combinatorial Theory B} 90 (2004) 257--307.

\bibitem{eps}
H.~Everett, C.M.H.~de~Figueiredo, C.~Linhares~Sales, F.~Maffray,
O.~Porto, B.A.~Reed.  Even pairs.  In: \emph{Perfect Graphs},
J.L.~Ram\'{\i}rez-Alfons\'{\i}n and B.A.~Reed eds., Wiley
Interscience, 2001, 67--92.

\bibitem{Far89}
M. Farber. On diameters and radii of bridged graphs. {\it Discrete
Math.}  73 (1989) 249--260.

\bibitem{Gol}
M.C.~Golumbic.  {\it Algorithmic Graph Theory and Perfect Graphs} (2nd
Edition).  North Holland, 2004.

\bibitem{GLS}
M.~Gr\"otschel, L.~Lov\'asz, A.~Schrijver.  The ellipsoid method and
its consequences in combinatorial optimization.  {\it Combinatorica} 1
(1981) 169--197.

\bibitem{Houg}
S.~Hougardy.  Even and odd pairs in linegraphs of bipartite graphs.
{\it Eur.  J. Comb.} 16 (1995) 17--21.

\bibitem{LMRT}
B.~L\'ev\^eque, F.~Maffray, B.~Reed, N.~Trotignon.  Coloring Artemis
graphs. {\it Theoretical Computer Science} 410 (2009) 2234--2240.

\bibitem{M}
F.~Maffray.  Coloring square-free Berge graphs with no odd~prism.
http://arxiv.org/abs/1502.03695.
 

\bibitem{MTa}
F.~Maffray, N.~Trotignon.  Algorithms for perfectly contractile
graphs.  {\it SIAM Journal on Discrete Mathematics} 19 (2005)
553--574.
 
\bibitem{MT}
F.~Maffray, N.~Trotignon. A class of perfectly contractile graphs.
{\it Journal of Combinatorial Theory B} 96 (2006) 1--19.

\bibitem{MU04}
K.~Makino, T.~Uno.  New algorithms for enumerating all maximal
cliques.  {\it Lecture Notes in Compute Science} 3111 (2004) 260--272.

\bibitem{PRR}
I.~Parfenoff, F.~Roussel, I.~Rusu.  Triangulated neighbourhoods in
$C_4$-free Berge graphs.  {\it Lecture Notes in Computer Science} 1665
(1999) 402--412.

\bibitem{TIAS}
S.~Tsukiyama, M.~Ide, H.~Ariyoshi, I.~Shirakawa.  A new algorithm for
generating all the maximal independent sets.  {\it SIAM J. Comput.} 6
(1977) 505--517.

 
 
\end{thebibliography}
\end{document}